\documentclass[a4paper,11pt]{article}
\usepackage{makeidx}
\usepackage[mathscr]{eucal}
\usepackage[dvips]{color}
\usepackage{amssymb, amsfonts, amsmath, amsthm, graphicx, cases, comment} 
\usepackage{here}
\usepackage{tasks}
\newtheorem{definition}{Definition}[section]
\newtheorem{proposition}{Proposition}[section]
\newtheorem{theorem}{Theorem}[section]

\newtheorem{lemma}[proposition]{Lemma}
\newtheorem{remark}{Remark}[section]

\newtheorem*{thank}{Acknowledgments}

\newcommand{\C}{\mathbb{C}}

\newcommand{\R}{\mathbb{R}}

\numberwithin{equation}{section}
\newcommand{\p}{\partial}

\date{}
\title{Threshold solutions for 
the 3D focusing cubic-quintic  
nonlinear Schr\"{o}dinger 
equation at low frequencies}

\author{Masaru Hamano, 
Hiroaki Kikuchi, Minami Watanabe}
\date{}
\begin{document}
\maketitle

\begin{abstract}
This paper addresses
the focusing cubic-quintic 
nonlinear Schr\"{o}dinger equation in 
three space dimensions. 
Especially, we study the 
global dynamics of 
solutions whose energy 
and mass equal to those of 
the ground state in the 
sprits of 
Duyckaerts and Merle~\cite{DM2009}. 
When we try to obtain 
the corresponding results of 
\cite{DM2009}, 
we meet several difficulties 
due to the cubic-quintic 
nonlinearity. 
We overcome them by 
using the one-pass theorem 
(no return theorem) developed 
by Nakanishi and 
Schlag~\cite{NS2012}. 
\end{abstract}


\section{Introduction}
In this paper, we consider 
 the following nonlinear 
Schr\"{o}dinger equation: 
\begin{equation} \label{nls}
i \p_{t} \psi + \Delta \psi 
+ |\psi|^{2} \psi + |\psi|^{4} \psi = 0
\qquad \mbox{in $\R \times\R^{3}$},  
\end{equation}
where 
$\Delta$ is the Laplace operator on $\R^{3}$. 
Several studies have been made on 
the asymptotic behavior of solutions 
to double power nonlinear 
Schr\"{o}dinger equations 
(see e.g. \cite{AM2022, AIKN2021, CC, 
FO2018, FH2021, Fuk2003, 
Hayashi2021, KLT2022, KMV2021, 
KOPV2017, LMR,  MZZ, Murphy2021, 
Ohta1995, OY2015, TVZ2007} and references therein). 
Here, we are concerned with global dynamics of 
solutions whose mass and 
energy equal to those of the ground state.

For any $\psi_{0} \in H^{1}(\R^3)$, 
there exists a unique solution $\psi$ in 
$C(I_{\max}; H^{1}(\R^3))$ with $\psi|_{t=0} 
= \psi_{0}$ for some interval 
$I_{\max} = (- T_{\max}^{-}, T_{\max}^{+}) \subset \R$, 
a maximal existence interval including $0$. 
We say that $\psi$ blows up in 
finite time if $T_{\max}^{+} < \infty$ or $T_{\max}^{-} < \infty$. 
The solution $\psi$ satisfies the following 
conservation laws of the mass and 
the energy in 
this order: 
\begin{equation} \label{eq-con}
\mathcal{M}(\psi(t)) = \mathcal{M}(\psi_{0}), \qquad 
\mathcal{E}(\psi(t)) = \mathcal{E}(\psi_{0}), 
\end{equation}
where 
\begin{equation*} 
\mathcal{M}(u) := \frac{1}{2} \|u\|_{L^{2}}^{2}, \qquad 
\mathcal{E}(u) := \frac{1}{2} \|\nabla u\|_{L^{2}}^{2}
- \frac{1}{4} \|u\|_{L^{4}}^{4}
- \frac{1}{6} \|u\|_{L^{6}}^{6}
\qquad \mbox{for $u \in H^{1}(\R^{3})$} 
\end{equation*} 

If, in addition, $\psi_{0} \in L^{2}
(\R^3, |x|^{2} dx)$, then the corresponding solution 
$\psi$ also belongs to $C(I_{\max}; L^{2}
(\R^{3}, |x|^{2} dx))$ and satisfies the so-called 
virial identity:
\begin{equation} \label{virial}
\begin{split}
\int_{\R^{3}} |x|^{2} |\psi(t, x)|^{2} dx 
& = \int_{\R^3} |x|^{2} |\psi_{0}(x)|^{2} dx 
+ 2 t \ \textrm{Im} \int_{\R^{3}} x \cdot \nabla \psi_{0}(x)
\overline{\psi_{0}(x)} dx \\
& \quad 
+ 16 
\int_{0}^{t} \int_{0}^{t^{\prime}} 
\mathcal{K}(\psi(t^{\prime \prime})) dt^{\prime \prime} 
dt^{\prime} \qquad \mbox{for any $t \in I_{\max}$}, 
\end{split}
\end{equation}
where 
\begin{equation*} 
\mathcal{K}(u):= \|\nabla u\|_{L^{2}}^{2} 
- \frac{3}{4}\|u\|_{L^{4}}^{4} 
- \|u\|_{L^{6}}^{6} 
\qquad \mbox{for $u \in H^{1}(\R^{3})$}. 
\end{equation*}
See e.g. Cazenave~\cite[Section 6.5]
{Cazenave} for details.

By a \textit{standing wave}, we mean a solution 
to \eqref{nls} of the form 
\begin{equation*}
\psi(t, x) = e^{i \omega t} Q_{\omega}(x)
\end{equation*} 
for some $\omega > 0$ and 
$Q_{\omega} \in H^{1}(\R^{3}) 
\setminus \{0\}$. 
Then, we see that 
$Q_{\omega}$ should solve 
the following semilinear elliptic equation:
\begin{equation} \label{sp}
- \Delta Q + \omega Q - |Q|^{2} Q 
- |Q|^{4} Q = 0
\qquad \mbox{in $\R^{3}$}. 
\end{equation}
If we define the action functional 
$\mathcal{S}_{\omega}$ by 
\begin{equation} \label{action}
\mathcal{S}_{\omega}(u) := 
\mathcal{E}(u) + \omega \mathcal{M}(u) 
\qquad \mbox{for $u \in H^{1}(\R^{3})$}, 
\end{equation}
then $\mathcal{S}_{\omega}^{\prime}(Q_{\omega}) 
= 0$ 
in $H^{-1}(\R^{3})$ if and only if 
$Q_{\omega} \in H^{1}(\R^{3})$ is a 
weak solution to \eqref{sp}. 
To seek a solution to \eqref{sp}, 
we consider the following 
minimization problem: 
	\begin{equation}\label{vari-v}
	m_{\omega} := 
	\inf\left\{ 
	\mathcal{S}_{\omega} (u) \colon 
	u \in H^{1}(\R^{3}) \setminus \{0\}, 
	\; \mathcal{K} (u) = 0
	\right\}. 
	\end{equation}
It is known that 
if there is a minimizer for 
$m_{\omega}$, 
it satisfies \eqref{sp}. 
Here we call $Q_{\omega}$ a 
\textit{ground state}
to \eqref{vari-v} provided 
$Q_{\omega}$ 
is a minimizer for $m_{\omega}$. 
Concerning the existence of a 
ground state, 
the following results hold:  
\begin{theorem}[\cite{AIKN, WW}]\label{thm-0}
There exists $\omega_{c} > 0$ such that 
$m_{\omega}$ has a ground state for  
$0< \omega < \omega_{c}$  
and no ground state for $\omega > \omega_{c}$.
\end{theorem}
\begin{remark}
We do not know whether the ground state 
exists or not at $\omega = \omega_c$. 
\end{remark}
There are several results on 
the global dynamics of solutions to 
nonlinear Schr\"{o}dinger equations.  
See e.g. 
\cite{AHI2022, AI2022, AM2022, AN2013, AIKN2021, CC, 
Dodson2019, 
DHR2008, DM2009, DOR, 
DR, GI2022, HR2008, KM2006, KMV2021, KOPV2017,
KV2010, LZ2009, 
NS2012, Y2022, YZZ2022} and references therein.
Let us recall some of them which are concerned with 
the following nonlinear Schr\"{o}dinger equations: 
\begin{equation} \label{nls-p}
i \p_{t} \psi + \Delta \psi + |\psi|^{p-1} \psi = 0
\qquad \mbox{in $\R \times\R^{d}$}. 
\end{equation}
where 
$d \in \mathbb{N}, 1 < p < 2^{*} - 1$ and 
$2^{*} = \frac{2d}{d -2}$. 
The equation \eqref{nls-p} is scale invariant. 
 More precisely, putting
 \begin{equation} \label{eq-scale}
\psi_{\lambda}(t, x) 
:= \lambda^{\frac{2}{p-1}} \psi(\lambda^{2} t, \lambda x) 
\qquad 
(\lambda > 0),  
\end{equation}
we see that if 
$\psi(t, x)$ satisfies \eqref{nls-p}, 
so does $\psi_{\lambda}$. 
The scaling \eqref{eq-scale} preserves 
the mass $\mathcal{M}$ and the corresponding 
energy when $p = 1 + 4/d$ and 
$p = (d+2)/(d-2)$, respectively.  
Thus, the exponent 
$p = 1 + 4/d$
is referred to as \lq\lq mass critical'' and 
$p = (d+2)/(d-2)$ as 
\lq\lq energy critical''.~\footnote
{ 
Note that the quintic power 
nonlinearity $|\psi|^{4} \psi$ in three space dimensions 
which is involved in \eqref{nls}
corresponds to the energy critical 
one}

It is known that \eqref{nls-p} 
has a stationary solution, which 
neither scatters~\footnote{Here, we say that a solution scatters if 
the solution converges to the one of the 
linear Schr\"{o}dinger equation} nor 
blows up. 
More precisely, 
when the energy critical case, 
\eqref{nls-p} has the 
following explicit static solution, 
	\[
	W(x) := 
	\left(1 + 
	\frac{|x|^{2}}{d(d-2)} 
	\right)^{- \frac{d-2}{2}}. 
	\]
The solution $W$ is called by 
\textit{Aubin-Talenti function}. 
Similarly, 
when $1 < p < (d+2)/(d-2)$, 
the equation 
\eqref{nls-p} also has a  
standing wave 
$\psi(t, x) = 
e^{i \omega t}R_{\omega}\; 
(\omega >0)$. 
Then, $R_{\omega}$ satisfies the 
following semilinear equation: 
\begin{equation}\label{sp-p}
- \Delta R + \omega R - |R|^{p-1}R = 0
\qquad \mbox{in $\R^{d}$}. 
\end{equation}

For the energy critical case 
$p = (d+2)/(d-2)$, 
Kenig and Merle~\cite{KM2006} 
employed the concentration-compactness and  
showed that the radial solution 
to \eqref{nls-p} whose energy is 
less than that of the Aubin-Talenti 
function $W$ 
blows up in finite time 
or scatters as $t \to \pm \infty$ 
for $d = 3, 4, 5$. 
Killip and Visan~\cite{KV2010}
extended the result of \cite{KM2006} 
for $d \geq 5$,  
removing the radial condition.  
Dodson~\cite{Dodson2019} 
obtained the corresponding 
result of ~\cite{KV2010} for 
$d = 4$. 

For the mass supercritical and 
the energy subcritical case 
$1 + 4/d < p < (d+2)/(d-2)$, 
Holmer and Roudenko~\cite{HR2008} 
considered the three dimensional cubic  
nonlinear Schr\"{o}dinger equation 
($d =p = 3$) and 
proved that 
the radial solution 
below the ground state
scatters or blows up in finite time. 
Then, Duyckaerts, Holmer and 
Roudenko~\cite{DHR2008} extended 
the result of \cite{HR2008} 
to non-radial $H^{1}$ initial data.  
Then, Akahori and 
Nawa~\cite{AN2013} and 
Fang, Xie and 
Cazenave~\cite{FXC2011} extended 
the result to general dimension and 
power nonlinearity.

Duyckaerts and Merle~\cite{DM2009} 
studied 
the threshold solution to the energy critical nonlinear 
Schr\"{o}dinger equations, that is, 
the solution whose energy equals 
to the Aubin-Talenti function 
for $d = 3, 4, 5$. 
They constructed special solutions 
$W^{\pm}$, 
which converge to the 
Aubin-Talenti function $W$ in the 
positive time direction while 
$W^{+}$ blows up and $W^{-}$ scatters in the negative 
time direction, respectively. 
They also classified 
the threshold solutions 
under the radial assumption. 
Li and Zhang~\cite{LZ2009} 
extended the result of 
\cite{DM2009} to the higher 
dimensions $d \geq 6$. 
Duyckaerts and 
Roudenko~\cite{DR2010} studied the 
threshold solution for 
the three dimensional cubic
nonlinear Schr\"{o}dinger equations. 
They also constructed special solutions 
and classify all solutions 
(not necessarily radially symmetric)
at the threshold level. 
Recently, Campos, Farah and 
Roudenko~\cite{CFR} 
generalized the result of 
Duyckaerts and 
Roudenko~\cite{DR2010} to 
any dimension and any power of 
the nonlinearity. 
They also considered the 
energy critical case and 
gave an alternative proof of 
the result of Li and Zhang~\cite{LZ2009}. 
See also \cite{AHI2022, AI2022, AM2022, 
DOR, GI2022, Murphy2021} for the threshold 
solutions to other 
nonlinear Schr\"{o}dinger equations. 

In this paper, we address the threshold 
solution to \eqref{nls}. 
To state our results, we put 
	\begin{align*}
	& 
	\mathcal{BA}_{\omega} 
	:= \left\{u \in 
	H_{\text{rad}}^{1}(\R^{3}) 
	\colon 
	\mathcal{S}_{\omega} (u) 
	= m_{\omega}, \; 
	\mathcal{M}(u) 
	= \mathcal{M}(Q_{\omega}) 
	\right\}, \\
	& 
	\mathcal{BA}_{\omega, +} 
	:= \left\{u \in 
	\mathcal{BA}_{\omega}  
	\colon 
	\mathcal{K}(u) > 0 
	\right\}, \\
	& 
	\mathcal{BA}_{\omega, -} 
	:= \left\{u \in 
	\mathcal{BA}_{\omega}  
	\colon 
	\mathcal{K}(u) < 0 
	\right\}, \\
	& 
	\mathcal{BA}_{\omega, 0} 
	:= \left\{u \in 
	\mathcal{BA}_{\omega} 
	\colon 
	\mathcal{K}(u) = 0 
	\right\}.
	\end{align*}
Clearly, we have 
$\mathcal{BA}_{\omega} 
= \mathcal{BA}_{\omega, -} 
\cup \mathcal{BA}_{\omega, 0} 
\cup \mathcal{BA}_{\omega, +}$. 
We see from Proposition \ref{thm-lu} below 
that 
	\begin{equation}\label{set-B0}
	\mathcal{BA}_{\omega, 0} 
	= \left\{ e^{i \theta} Q_{\omega} 
	\colon \theta \in \R 
	\right\}. 
	\end{equation}
In addition, we can easily find that 
the sets $\mathcal{BA}_{\omega, \pm}$ 
and $\mathcal{BA}_{\omega, 0}$ are 
invariant under the flow of \eqref{nls} 
(see e.g. Lemma \ref{inv-} below). 
Then, 
by a similar argument to 
\cite{DM2009}, 
we can construct the 
following special solutions to 
\eqref{nls}:
\begin{theorem}\label{main-sp}
There exists a sufficiently small 
$\omega_{*}>0$ such that 
for $\omega \in (0, \omega_{*})$, 
\eqref{nls} has two radial solutions 
$Q_{\omega}^{+} \in
 \mathcal{BA}_{\omega, +} $ and 
$Q_{\omega}^{-} \in 
\mathcal{BA}_{\omega, -}$ 
satisfying 
the following:
\begin{enumerate}
\item[\textrm{(i)}] 
$Q_{\omega}^{\pm}$ exists 
on $[0, \infty)$, and 
there exist constants 
$e_{\omega}, C_{\omega}> 0$
such that 
\[
\text{dist}_{H^{1}}
(Q_{\omega}^{\pm}(t), 
\mathcal{O}(Q_{\omega})) \leq C_{\omega} 
e^{- e_{\omega} t} 
\qquad \mbox{for all $t \geq 0$},  
\]
where 
	\begin{align*}
{\rm dist}_{H^{1}}\big(u, \mathscr{O}(Q_{\omega})\big)
:=
\inf_{\theta \in \mathbb{R}}\|u -e^{i\theta}Q_{\omega}\|_{H^{1}}.
\end{align*}
\item[\textrm{(ii)}] 
$\mathcal{K}(Q_{\omega}^{-}) 
< 0$ and the negative 
time of existence of 
$Q_{\omega}^{-}$ is 
finite. 
\item[\textrm{(iii)}] 
$\mathcal{K}(Q_{\omega}^{+}) 
> 0$, 
$Q_{\omega}^{+}$ exists on 
$(-\infty, \infty)$ and scatters 
for negative time, that is, 
there exists $\phi_{-} \in H^{1}(\R^{3})$ such that 
  \[
  \lim_{t \to -\infty} \|Q_{\omega}^{+}(t) 
  - e^{i t \Delta}\phi_{-}\|_{H^{1}} = 0.
  \] 
\end{enumerate}
\end{theorem}
In what follows, we say that 
$\psi = \phi$ 
up to the symmetries if  
there exists $t_0 \in \R$ and 
$\theta_0$ such that 
	\[
	\psi(t, x) 
	= e^{i \theta_0} 
	\phi(t + t_{0}, x) 
	\qquad 
	\mbox{or}
	\qquad 
	\psi(t, x) 
	= e^{i \theta_0} 
	\overline{\phi} (- t + t_0, x). 
	\]
Our main result is as follows: 
\begin{theorem}\label{main-class}
Let $\omega_{*} > 0$ be a 
constant given in Theorem 
\ref{main-sp} and 
$\psi$ be a radial 
solution to 
\eqref{nls} with 
$\psi|_{t = 0} = \psi_{0} \in 
\mathcal{BA}_{\omega}$ 
for $\omega \in (0, \omega_*)$. 
Then, the following holds: 
\begin{enumerate}
\item[\textrm{(i)}] 
If $\psi_{0} \in \mathcal{BA}_{\omega, -}$, 
then either $\psi$ blows up in finite 
time or $\psi = Q_{\omega}^{-}$ 
up to the symmetries. 
\item[\textrm{(ii)}] 
If $\psi_{0} \in 
\mathcal{BA}_{\omega, 0}$, 
then $\psi = 
e^{i \omega t} Q_{\omega}$ 
up to the symmetries. 
\item[\textrm{(iii)}] 
If $\psi_{0} \in \mathcal{BA}_{\omega, +}$,
then either $\psi$ scatters or 
$\psi = Q_{\omega}^{+}$ 
up to the symmetries. 
\end{enumerate}
\end{theorem}
\begin{remark}
\begin{enumerate}
\item[\textrm{(i)}]
Nakanishi and 
Schlag~\cite{NS2012} proved that  
the solutions whose energy is 
slightly larger than that of the 
ground state is 
classified into 9 sets 
(combination of 
blows up, scattering and 
trapped by the ground state 
generated by the phase for 
$t > 0$ and $t < 0$). 
See also \cite{AIKN2021, DR, 
KMV2021, NakanishiJMSC, 
NakanishiCMP} for the global dynamics above the ground state. 
In particular, 
it was studied in \cite{AIKN2021}
that the behavior of 
the solutions $\psi$ to \eqref{nls} 
with $\psi|_{t = 0} = \psi_0$ 
satisfying 
$ \mathcal{S}_{\omega}(\psi_{0}) 
< m_{\omega} + \varepsilon$ 
are also classified into the 9 sets.  
However, it seems that 
from the result of \cite{AIKN2021}, 
we could not determine 
the behavior of solutions 
by the initial data as in Theorem 
\ref{main-class}. 
Another difference between 
the result of \cite{AIKN2021} and ours 
is that we obtain a kind of 
uniqueness of solution which converges 
to the orbit of the ground state 
(see Proposition \ref{prop-unique} below). 
\item[\textrm{(ii)}] 
We may extend our results 
to general dimensions and power nonlinearities by using the argument 
of \cite{CFR}. 
However, for the simplicity of 
our presentation, we restrict ourselves 
to three space dimension and 
cubic-quintic nonlinearity. 
\end{enumerate} 
\end{remark}
The proof of Theorem \ref{main-class} 
is based on that of 
\cite{DM2009, DR2010}. 
However,  
it seems that 
due to the cubic and quintic 
nonlinearities, 
some part of 
the argument in \cite{DM2009, 
DR2010} 
does not work for our equation 
\eqref{nls}. 
For example, in \cite{DR2010}, 
a Cauchy-Schwarz type inequality 
plays an important role (see 
\cite[Claim 5.4]{DR2010} in detail). 
In contrast, it seems difficult to 
obtain a corresponding inequality for 
our equation \eqref{nls}. 
To overcome the difficulty, 
 we employ the one-pass theorem 
 (no return theorem)
which was introduced by Nakanishi and 
Schlag~\cite{NS2012} for the equation \eqref{nls-p} 
with $d = p = 3$. 
Roughly speaking, 
one-pass theorem states that 
if a solution moves away 
from a neighborhood of the ground states, then the solution never 
return to the neighborhood. 
We employ the one-pass theorem 
to prove that if a threshold 
solution neither blows up nor 
scatters, the solution converges 
to the ground state exponentially 
(see Propositions \ref{prop-conv-n} 
and \ref{prop-conv-p}). 
\begin{remark}
\begin{enumerate}
\item[\textrm{(i)}]
The reason why we need the 
radially symmetry for solutions 
is due to the one-pass theorem. 
Indeed, a kind of 
Ogawa-Tsutsumi's saturated virial identity was used for the proof of 
the one-pass theorem. 
Except for the theorem,  
we do not require 
the condition. 
\item[\textrm{(ii)}]
Recently, Ardila and Murphy~\cite{AM2022} studies 
the threshold 
solutions to the following cubic-quintic nonlinear 
Schr\"{o}dinger equation: 
\begin{equation} \label{nls-}
i \p_{t} \psi + \Delta \psi 
+ |\psi|^{2} \psi - |\psi|^{4} \psi = 0
\qquad \mbox{in $\R \times\R^{3}$}.   
\end{equation}
Note that the quintic power 
nonlinearity is defocusing, which is different 
from our equation \eqref{nls} 
and any solutions to \eqref{nls-} 
are global ($T_{\max}^{\pm} = \infty$). 
They also classified the threshold solutions, 
which are not necessary radially symmetric,  
to \eqref{nls-}. 
Let $\psi$ be the 
threshold solution whose sign of the virial 
functional (the one corresponding to $\mathcal{K}$) is 
positive. 
Then, they showed that the solution 
either $\psi$ scatters in both time directions and 
coincide with a special solution.      
To this end, they employed the modulation analysis and 
the concentration-compactness method. 
Their method might work for our equation \eqref{nls}. 
However, we would like to stress that 
we can study the threshold solutions 
which blow up or scatter in a unified way by using 
the one-pass theorem. 
\end{enumerate}
\end{remark}

This paper is organized as follows. 
In Section 2, we recall 
several properties of the ground state 
and its linearized operator. 
In Section 3, we recall the one-pass 
theorem, which was proved in 
\cite{AIKN2021}. 
In Sections 4 and 5, 
we study dynamics of solutions 
which start from the sets 
$\mathcal{BA}_{\omega, \pm}$, 
respectively.  
In Section 6, we give a sketch of 
the proof of Theorem \ref{main-sp}. 
In Section 7, we obtain the uniqueness 
of the special solution and give the 
proof of Theorem \ref{main-class}. 
In Appendix, we give the proof of 
a convergence property which we admit in 
Section 4.  

\noindent
\textbf{Notation. 
}
\begin{enumerate}
\item[\textrm{(i)}]
We use $(\cdot, \cdot)_{L^{2}}$ to 
denote the inner product in $L^{2}(\R^{3})$:
  \[
  (u, v)_{L^{2}} := \int_{\R^{3}} u(x) 
  \overline{v(x)} dx \qquad 
  \mbox{for all $u, v \in L^{2}(\R^{3})$}. 
  \]
\item[\textrm{(ii)}]
We also use $L^{2}_{\text{real}}(\R^{3})$ 
to denote the real Hilbert space of complex-valued 
functions in $L^{2}(\R^{3})$ which is equipped with 
the inner product 
  \[
  (u, v)_{L^{2}_{\text{real}}}:= 
  \text{Re}\; \int_{\R^{3}} u(x) \overline{v(x)} dx 
   \qquad 
  \mbox{for all $u, v \in L_{\text{real}}^{2}(\R^{3})$}.
  \]
\item[\textrm{(iii)}]
We use $\langle \cdot, \cdot \rangle_{H^{-1}, H^{1}}$ 
to denote the duality pair of $u \in H^{1}(\R^{3})$ 
and $v \in H^{1}(\R^{3})$: 
   \[
  \langle u, v\rangle_{H^{-1}, H^{1}} 
  := ((1 - \Delta)^{- \frac{1}{2}} u, 
  (1 - \Delta)^{- \frac{1}{2}} v)_{L^{2}_{\text{real}}}\qquad 
  \mbox{for all $u \in H^{1}(\R^{3})$ and 
  $v \in H^{-1}(\R^{3})$}. 
  \]
\end{enumerate}
\section{Properties of ground state and 
its linearized operator}
In this section, we recall several 
properties of the ground state 
and its linearized operator, 
which are mainly 
obtained in \cite{AIKN2021}.  
First, we recall that the 
uniqueness of ground state 
and that the following 
slope condition holds:  
\begin{proposition}[Proposition 2.0.4 
of \cite{AIKN2021}]
\label{thm-lu}
The following properties hold: 
\begin{enumerate}
\item[\textrm{(i)}]  
There exists $\omega_{1} > 0$ such that for 
$\omega \in (0, \omega_{1})$, 
the positive radial 
ground state $Q_{\omega}$ is 
unique up to phase. 
Namely, if 
$u \in H^{1}(\R^3)$ satisfies 
$\mathcal{S}_{\omega}(u) = m_{\omega}$ and 
$\mathcal{K}(u) = 0$, then we have 
$u = e^{i \theta} Q_{\omega}$ for some 
$\theta \in \R$. 
\item[\textrm{(ii)}]
The mapping $\omega \in (0, 
\omega_{1}) \mapsto Q_{\omega} 
\in H^{1}(\R^{3})$ is continuously 
differentiable, 
\item[\textrm{(iii)}] 
\begin{equation*}
\frac{d}{d \omega}
\mathcal{M}(Q_{\omega}) 
= (Q_{\omega}, \p_{\omega} 
Q_{\omega})_{L_{\text{real}}
^{2}}< 0 
\qquad \mbox{for $\omega \in (0, \omega_1)$}. 
\end{equation*}
\end{enumerate}
\end{proposition} 
Let $\psi$ be a solution to \eqref{nls}. 
We consider the following decomposition of the form 
\begin{equation}\label{eq-dec1}
\psi(t,x)=e^{i\theta(t)}\big(Q_{\omega}(x)+\eta(t,x)\big),
\end{equation} 
where $\theta(t)$ is a 
function of $t\in I_{\max}$ to be chosen later 
(see \eqref{eq-dec18} below)
and $\eta$ is the remainder. 
Let $\eta_{1} = \textrm{Re} \; \eta$ 
and 
$\eta_{2} = \textrm{Im} \; \eta$. 
We will identify $\C$ and $\R^{2}$ and 
consider $\eta = \eta_{1} + i \eta_{2}$ as 
an element 
$\begin{pmatrix}
\eta_{1} \\
\eta_{2}
\end{pmatrix}$ of $\R^{2}$. 
Then, $\eta$ satisfies 
\begin{equation}\label{eq-dec7}
\begin{split}
\frac{\partial \eta}{\partial t}(t)
&=
-i\mathscr{L}_{\omega} \eta(t)
-
i\Big\{ \frac{d \theta}{dt}(t)- \omega \Big\}(Q_{\omega}+\eta(t)) 
+
N_{\omega}(\eta(t)),
\end{split}
\end{equation}
where 
\begin{equation*} 
- i \mathcal{L}_{\omega}
= 
\begin{pmatrix}
0 & L_{\omega, -} \\
- L_{\omega, +} & 0
\end{pmatrix}, 
\end{equation*}
and the self-adjoint operators 
$L_{\omega, \pm}$ and 
the reminder $N_{\omega}(\eta)$ are defined by 
\begin{equation}\label{sta-eq54}
L_{\omega, +} := - \Delta + \omega 
- 3 Q_{\omega}^{2} - 5 Q_{\omega}^{4}, 
\qquad 
L_{\omega, -} := - \Delta + \omega 
- Q_{\omega}^{2} - Q_{\omega}^{4}, 
\end{equation}
\begin{equation}\label{eq-dec8}
\begin{split}
N_{\omega}(\eta)&:=
N_{\omega, 1}(\eta) + N_{\omega, 2}(\eta)
\end{split}
\end{equation}
Here, $N_{\omega, 1}(\eta)$ 
and $N_{\omega, 2}(\eta)$ are 
defined by 
\begin{equation*}
N_{\omega, 1} (\eta) 
:= i \left(|Q_{\omega}+ \eta|^{2}\big( Q_{\omega}+\eta \big) 
- Q_{\omega}^{3}
-2 Q_{\omega}^{2}\eta 
- Q_{\omega}^{2}\overline{\eta}
\right), 
\end{equation*}
\begin{equation*}
\begin{split}
N_{\omega, 2} (\eta)
&:= i \left(
|Q_{\omega}+\eta|^{4}\big( Q_{\omega}+\eta \big) 
- Q_{\omega}^{5}
- 3 Q_{\omega}^{4}
\eta 
- 2Q_{\omega}^{4} 
\overline{\eta} 
\right).
\end{split}
\end{equation*}
We note that 
for any $u,v \in H^{1}(\R^{3})$, 
\begin{equation}\label{eq-dec2}
\begin{split}
\big[ \mathcal{S}_{\omega}''(Q_{\omega})u \big]v 
&= 
\langle \mathscr{L}_{\omega}u, v 
\rangle_{H^{-1},H^{1}}. 
\end{split}
\end{equation} 
Moreover, since $Q_{\omega}$ is a solution to \eqref{sp}, 
we can verify that 
\begin{align}
\label{eq-dec4}
&\mathscr{L}_{\omega}Q_{\omega}=L_{\omega,+}Q_{\omega} 
= - 2 Q_{\omega}^{3}- 
4Q_{\omega}^{5},
\\[6pt]
\label{eq-dec5}
&\mathscr{L}_{\omega} (iQ_{\omega}) 
= \mathscr{L}_{\omega}
\begin{pmatrix}
0 \\
Q_{\omega}
\end{pmatrix}
=L_{\omega,-}Q_{\omega}=0,
\\[6pt]
\label{eq-dec6}
&
\mathscr{L}_{\omega} \partial_{\omega}Q_{\omega}
=
L_{\omega,+}\partial_{\omega}Q_{\omega} =- Q_{\omega}.
\end{align}

\begin{theorem}[Proposition  4.0.1, 
Lemmas B.0.1 and B.0.2 of \cite{AIKN2021}]\label{thm2-1}
Let $\sigma(- i \mathcal{L}_{\omega})$ 
be the spectrum of the operator 
$- i \mathcal{L}_{\omega}$ 
and $\sigma_{\text{ess}}
(- i \mathcal{L}_{\omega})$ 
be its essential spectrum. 
Then, we have the following: 
\begin{enumerate}
\item[\textrm(i)] 
$\sigma_{\text{ess}}
(- i \mathcal{L}_{\omega}) 
= \left\{
i \xi \colon \xi \in \R, |\xi| \geq \omega
\right\}$. 
\item[\textrm(ii)] 
There exists $\omega_{2} \in (0, \omega_{1})$ such that 
for $\omega \in (0, \omega_{2})$,  
$- i \mathcal{L}_{\omega}$ 
has positive and negative 
eigenvalues 
$e_{\omega}$ and 
$-e_{\omega}$ with 
eigenfunction $\mathcal{Y}_{\omega, +}$ 
and $\mathcal{Y}_{\omega, -}$, 
respectively. 
Furthermore, $\textrm{Ker} \; 
(- i \mathcal{L}_{\omega}) 
= \textrm{Span}\; 
\left\{ 
i Q_{\omega}, 
\partial_{x_{i}}Q_{\omega} 
\; \text{for $i = 1, 2, 3$}
\right\}$. 
\end{enumerate}
\end{theorem}

It is known that 
$\overline{\mathcal{Y}_{\omega, +}} 
= \mathcal{Y}_{\omega, -}$. 
Then, we put 
\begin{equation}\label{eq-egf1}
\mathcal{Y}_{\omega, 1}
:=\frac{\mathcal{Y}_{\omega, +}+\mathcal{Y}_{\omega, -}}{2}
=\textrm{Re}\;[\mathcal{Y}_{\omega, +}],
\qquad 
\mathcal{Y}_{\omega, 2}
:=
\frac{\mathcal{Y}_{\omega, +} 
- \mathcal{Y}_{\omega, -}}{2i}
=
\textrm{Im}\; [\mathcal{Y}_{\omega, +}], 
\end{equation}
that is, 
  \[
  \mathcal{Y}_{\omega, +} = 
  \begin{pmatrix}
  \mathcal{Y}_{\omega, 1} \\
  \mathcal{Y}_{\omega, 2}
  \end{pmatrix}. 
  \]
The equation 
$- i \mathcal{L}_{\omega} \mathcal{Y}_{\omega, +} 
= e_{\omega} \mathcal{Y}_{\omega, +}$ is 
equivalent to 
\begin{equation}\label{eq-pre1}
\begin{cases}
L_{\omega, +} 
\mathcal{Y}_{\omega, 1} 
= - e_{\omega} \mathcal{Y}_{\omega, 2}, & \\
L_{\omega, -} 
\mathcal{Y}_{\omega, 2} 
= e_{\omega} 
\mathcal{Y}_{\omega, 1}, 
&
\end{cases}
\end{equation}
Concerning the eigenfunctions 
$\mathcal{Y}_{\omega, 1}$ and 
$\mathcal{Y}_{\omega, 2}$, 
we know the following:
\begin{lemma}[Lemma 4.0.7 of 
\cite{AIKN2021}]\label{lem-egf2}
Let $\omega \in (0, \omega_{2})$. 
We have the following orthogonalities:
\begin{equation} \label{eq-egf6}
\big(Q_{\omega},\mathcal{Y}_{\omega, 1} \big)_{L^{2}}
=
\big( \partial_{\omega}Q_{\omega}, \mathcal{Y}_{\omega, 2} \big)_{L^{2}}
=0.
\end{equation}
Furthermore, we have 
\begin{equation}\label{eq-egf5}
\big( \mathcal{Y}_{\omega, 1}, \mathcal{Y}_{\omega, 2} \big)_{L^{2}} 
> 0
\end{equation}
and 
\begin{equation}\label{eq-egf7}
\big(Q_{\omega}, \mathcal{Y}_{\omega, 2} \big)_{L^{2}}\neq 0.
\end{equation}
\end{lemma}
The relation \eqref{eq-egf7} in Lemma \ref{lem-egf2} allows us to choose $\mathcal{Y}_{\omega, 2}$ so that 
\begin{equation}\label{eq-egf10}
(Q_{\omega},\mathcal{Y}_{\omega, 2})_{L^{2}}<0.
\end{equation} 
Note that 
since $Q_{\omega}$ is positive (especially, real-valued) and 
$\overline{\mathcal{Y}_{\omega, +}} 
= \mathcal{Y}_{\omega, -}$, 
we have, by \eqref{eq-egf1} and \eqref{eq-egf6}, that 
  \begin{equation}\label{eq-egf8} 
  (Q_{\omega}, \mathcal{Y}_{\omega, \pm})_{L^{2}_{\text{real}}} 
  = (Q_{\omega}, \mathcal{Y}_{\omega, 1})_{L^{2}} 
  = 0. 
  \end{equation}
In addition, we need the following 
technical lemma:
\begin{lemma}[Lemma 4.0.8 of
\cite{AIKN2021}] \label{tech-lem1} 
There exists a frequency 
$\omega_{3} \in (0,\omega_{2})$ 
such that for any $\omega \in (0,\omega_{3})$, 
\begin{equation}\label{eq-ejec4}
\frac{e_{\omega}}{2} 
\big| ( Q_{\omega}, \mathcal{Y}_{\omega, 2} )_{L^{2}}\big|
\ge 4
\big| (Q_{\omega}^{5}, \mathcal{Y}_{\omega, 1} )_{L^{2}}\big|.
\end{equation}
\end{lemma}

\section{One-pass theorem}
This section is devoted to the one-pass theorem (no 
return theorem)
which was first 
obtained by Nakanishi and Schlag~\cite{NS2011, NS2012} 
for the cubic nonlinear 
Klein-Gordon and Schr\"{o}dinger equations 
in three space dimensions. 
In \cite{AIKN2021}, the one-pass 
theorem for the double power nonlinear 
Schr\"{o}dinger equations was proved.  
Here, we recall the set-up and the one-pass theorem 
of \cite{AIKN2021}.  
\subsection{Symplectic 
decomposition and parameter choice}
\label{sec-dec}
For a positive radial ground state 
$Q_{\omega}$ to \eqref{sp} and a solution $\psi$ to \eqref{nls}, 
we consider the decomposition \eqref{eq-dec1}. 
We will work in the symplectic space $(L^{2}(\mathbb{R}^{3}), \Omega)$, 
where $\Omega$ is the symplectic form defined by 
\begin{equation*}
\Omega(f,g)
:=\big( f, i g\big)_{L_{\text{real}}^{2}}= \textrm{Im}\; \int_{\R^{3}}f(x)\overline{g(x)}\,dx.
\end{equation*}
\par 
We apply the \lq\lq symplectic decomposition'' corresponding to the discrete modes of 
$i\mathscr{L}_{\omega}$ to the 
remainder $\eta$ in \eqref{eq-dec1} 
and determine the function 
$\theta(t)$ in \eqref{eq-dec1}. 
\par 
We assume that 
$\mathcal{Y}_{\omega, +}$ 
and $\mathcal{Y}_{\omega, -}$ 
are normalized in the following sense: 
\begin{equation}
\label{eq-dec10}
\Omega\big(
\mathcal{Y}_{\omega, +}, \mathcal{Y}_{\omega, -} \big)
=1,
\quad 
\Omega\big( \mathcal{Y}_{\omega, -}, \mathcal{Y}_{\omega, +} \big)
=-1.
\end{equation}
We can easily find that 
\begin{equation*}
\Omega\big(f, f \big)
= 
\textrm{Im}\; \int_{\R^{3}} |f|^{2} dx 
= 0 
\end{equation*}
for all $f \in L^{2}(\R^{3})$. 
Furthermore, it follows from 
\eqref{eq-dec5}, \eqref{eq-dec6}, 
\eqref{eq-egf8} and $\mathcal{L}_{\omega}
\mathcal{Y}_{\omega, 
\pm}=\pm i e_{\omega} 
\mathcal{Y}_{\omega, \pm}$ that 
\begin{equation}\label{eq-dec12}
\begin{split}
& \quad \Omega\big( iQ_{\omega}, 
\mathcal{Y}_{\omega, \pm} \big)
= \big(i Q_{\omega}, i \mathcal{Y}_{\omega, \pm}\big)_{L^{2}_{\text{real}}}
= \pm \frac{1}{e_{\omega}} 
\big(i Q_{\omega},  
\mathcal{L}_{\omega} 
\mathcal{Y}_{\omega, \pm}\big)_{L^{2}_{\text{real}}} \\
& = \pm \frac{1}{e_{\omega}} 
\big(\mathcal{L}_{\omega}(i Q_{\omega}), 
\mathcal{Y}_{\omega, \pm}\big)_{L^{2}_{\text{real}}} 
= 0, 
\end{split}
\end{equation}
\begin{equation}\label{eq-dec12-1}
\begin{split}
& \quad 
\Omega\big( \partial_{\omega}Q_{\omega}, 
\mathcal{Y}_{\omega, \pm} \big) 
= \pm \frac{1}{e_{\omega}} 
(\p_{\omega} Q_{\omega}, \mathcal{L}_{\omega} 
\mathcal{Y}_{\omega, \pm})_{L^{2}_{\text{real}}} 
= \pm \frac{1}{e_{\omega}} 
(\mathcal{L}_{\omega} 
\p_{\omega} Q_{\omega}, 
\mathcal{Y}_{\omega, \pm})_{L^{2}_{\text{real}}} \\
& 
= \mp \frac{1}{e_{\omega}} 
(Q_{\omega}, \mathcal{Y}_{\omega, \pm})_{L^{2}_{\text{real}}} 
=0.
\end{split}
\end{equation}

Now, we expand the remainder $\eta(t)$ 
in the decomposition \eqref{eq-dec1} 
by $-i\mathscr{L}_{\omega}$: 
\begin{equation}\label{eq-dec13}
\eta(t)
=
\lambda_{+}(t)\mathcal{Y}_{\omega, +} + 
\lambda_{-}(t)\mathcal{Y}_{\omega, -}+ a(t) iQ_{\omega}+b(t)\partial_{\omega}Q_{\omega}
+\gamma(t),
\end{equation}
where 
\begin{equation}\label{eq-dec14}
\Omega\big( \gamma(t), \mathcal{Y}_{\omega, \pm} \big)
=
\Omega\big( \gamma(t), iQ_{\omega} \big)
=
\Omega\big( \gamma(t), \partial_{\omega}Q_{\omega} \big)
=0.
\end{equation}
We see from Proposition \ref{thm-lu} {\rm (iii)} 
and \eqref{eq-dec10}--\eqref{eq-dec14} 
that the coefficients are as follows: 
\begin{align}\label{eq-dec15}
& 
\lambda_{+}(t)
=
\Omega(\eta(t), 
\mathcal{Y}_{\omega, -}), 
\qquad 
\lambda_{-}(t)
=
- \Omega(\eta(t), 
\mathcal{Y}_{\omega, +}), 
\qquad 
\\[6pt]
\label{eq-dec16}
&a(t)=
\frac{\Omega(\eta(t), \partial_{\omega}Q_{\omega})}{(Q_{\omega},
\partial_{\omega}Q_{\omega})_{L_{\text{real}}^{2}}},
\qquad 
b(t)=
-\frac{\Omega(\eta(t),iQ_{\omega})}{(Q_{\omega},
\partial_{\omega}Q_{\omega})_{L_{\text{real}}^{2}}}.
\end{align}
We require that $a(t)\equiv 0$. 
To this end, we choose the function $\theta(t)$ in \eqref{eq-dec1} so that
~\footnote{
Let 
$\psi(t) = |\psi(t)|e^{i \alpha(t)}\; 
(\alpha(t) \in \R/ 2\pi \mathbb{Z})$. 
It suffices to choose $\theta(t)$ 
so that $\theta(t) 
= \alpha(t)$. 
Then, $\Omega(e^{-i\theta(t)}\psi(t),\partial_{\omega}Q_{\omega}) = (|\psi(t)|, i \p_{\omega} 
Q_{\omega})_{L_{\text{real}}^{2}} = 0$}
\begin{equation}\label{eq-dec18}
\Omega(e^{-i\theta(t)}\psi(t),\partial_{\omega}Q_{\omega})\equiv 0. 
\end{equation}
Then, it follows from 
\eqref{eq-dec18}, \eqref{eq-dec1} and 
$\Omega(Q_{\omega},\partial_{\omega}Q_{\omega})=0$ 
that 
\begin{equation} \label{eq-dec20}
\begin{split}
0 
\equiv 
\Omega(e^{-i\theta(t)}\psi(t),\partial_{\omega}Q_{\omega}) 
\equiv 
\Omega(Q_{\omega} + \eta(t), \p_{\omega} Q_{\omega}) 
\equiv 
\Omega(\eta(t),\partial_{\omega}Q_{\omega}). 
\end{split}
\end{equation}
This together with 
\eqref{eq-dec16} implies that 
$a(t)\equiv 0$. 
Furthermore, since 
\begin{equation*}
\mathcal{M}(\psi)
=\mathcal{M}(Q_{\omega})+\mathcal{M}(\eta(t))
+\big(Q_{\omega}, \eta(t)\big)_{L_{\text{real}}^{2}},
\end{equation*}
the condition $\mathcal{M}(\psi)= \mathcal{M}(Q_{\omega})$ 
implies that for any $t\in I_{\max}$, 
\begin{equation}\label{eq-uni2-2}
\big(Q_{\omega},\eta(t) \big)_{L_{\text{real}}^{2}}
= 
-\mathcal{M}(\eta(t)).
\end{equation} 
This together with 
$(Q_{\omega}, \p_{\omega} Q_{\omega})_{L^{2}_{\text{real}}} < 0$ 
yields that 
 \begin{equation}\label{eq-dec19}
\begin{split}
( e^{-i\theta(t)}\psi(t), 
\partial_{\omega}
Q_{\omega} )_{L_{\text{real}}^{2}} 
& = (Q_{\omega}, \p_{\omega} Q_{\omega})_{L^{2}_{\text{real}}} 
+ (\eta(t), \p_{\omega} Q_{\omega})_{L^{2}_{\text{real}}} \\
& = (Q_{\omega}, \p_{\omega} Q_{\omega})_{L^{2}_{\text{real}}}  
- M(\eta(t)) <0
\end{split}
\end{equation}
as long as $\mathcal{M}(\eta)$ is small. 
In what follows, 
we assume that $\psi$ satisfies \eqref{eq-dec18}, 
\eqref{eq-dec19} and 
$\mathcal{M}(\psi)= \mathcal{M}(Q_{\omega})$.

From \cite[Section 4.0.3]{AIKN2021}, 
ordinary differential equations for $\theta$, $\lambda_{-}$ and $\lambda_{+}$ under the condition 
$\mathcal{M}(\psi)= \mathcal{M}(Q_{\omega})$ are 
following: 
\begin{equation}\label{eq-mod4}
\Big\{ \frac{d \theta}{dt}(t)- \omega \Big\}
\big( Q_{\omega}+\eta(t), \partial_{\omega}Q_{\omega} \big)_{L_{\text{real}}^{2}}
=-\mathcal{M}(\eta(t))
+
\big( N_{\omega}(\eta(t)), \partial_{\omega}Q_{\omega} \big)_{L_{\text{real}}^{2}}.
\end{equation}
\begin{align}
\label{eq-mod5}
&\frac{d \lambda_{+}}{d t}(t)
=
e_{\omega} \lambda_{+}(t) 
-
\big(
\Big\{ \frac{d \theta}{dt}(t)- \omega \Big\}\eta(t) 
-N_{\omega}(\eta(t)), \, \mathcal{Y}_{\omega, -} \big)_{L_{\text{real}}^{2}}, 
\\[6pt]
\label{eq-mod6}
&
\frac{d \lambda_{-}}{d t}(t)
=
- e_{\omega} \lambda_{-}(t)
+ 
\big( \Big\{ \frac{d \theta}{dt}(t)- \omega \Big\} \eta(t) 
-N_{\omega}(\eta(t)), \, \mathcal{Y}_{\omega, +} \big)_{L_{\text{real}}^{2}}. 
\end{align}
\subsection{Linearized energy norm}\label{sec-energy}
In this subsection, we 
introduce the ``linearized energy norm'' 
for the remainder $\eta$ in the decomposition \eqref{eq-dec1}. 
Put 
\begin{equation}\label{eq-energy10}
\Gamma(t) :=b(t)\partial_{\omega}Q_{\omega}+\gamma(t).
\end{equation}
Since $a(t)\equiv 0$ (see \eqref{eq-dec18}), 
the decomposition \eqref{eq-dec13} is reduced to  
\begin{equation}\label{eq-energy1}
\eta(t)
=
\lambda_{+}(t)
\mathcal{Y}_{\omega, +} + 
\lambda_{-}(t)
\mathcal{Y}_{\omega, -}
+ \Gamma(t).
\end{equation}
Then, we have the following relationship:
\begin{lemma}[Lemma 4.0.2 of \cite{AIKN2021}]
\label{lem-energy1}
The function $\Gamma$ in the decomposition \eqref{eq-energy1} satisfies 
\begin{equation}\label{eq-energy0}
\langle \mathscr{L}_{\omega} \Gamma(t), \Gamma(t)\rangle_{H^{-1},H^{1}}
\sim
\|\Gamma (t) \|_{H^{1}}^{2} 
\end{equation}
for all $t\in I_{\max}$.  
\end{lemma}
We also recall that 
\begin{equation}\label{eq-energy3}
\begin{split}
\mathcal{E}(\psi)-\mathcal{E}(Q_{\omega})
= - e_{\omega} \lambda_{+}(t)
\lambda_{-}(t)
+
\frac{1}{2}
\langle 
\mathscr{L}_{\omega}\Gamma(t), 
\Gamma(t) \rangle_{H^{-1},H^{1}}
+ O(\, \|\eta(t)\|_{H^{1}}^{3} \,).
\end{split}
\end{equation}
See \cite[(4.69)]{AIKN2021}.
Defining the linearized energy norm 
$\|\eta(t)\|_{E}$ by 
\begin{equation}\label{eq-energy4}
\|\eta(t) \|_{E}^{2}
:=
\frac{e_{\omega}}{2} 
\big( \lambda_{+}^{2}(t) + 
\lambda_{-}^{2}(t)\big)
+
\frac{1}{2}
\langle 
\mathscr{L}_{\omega}\Gamma(t), 
\Gamma(t) \rangle_{H^{-1},H^{1}}, 
\end{equation}
we have, by \eqref{eq-energy3}, that 
\begin{equation}\label{eq-energy5}
\mathcal{E}(\psi)-\mathcal{E}(Q_{\omega})
+
\frac{e_{\omega}}{2} 
\big(\lambda_{+}(t) 
+ \lambda_{-}(t)\big)^{2}
-
\|\eta(t)\|_{E}^{2}
=O( \, \|\eta(t)\|_{H^{1}}^{3} \,).
\end{equation}
Then, we can also have the following lemma: 
\begin{lemma}[Lemma 4.0.3 of \cite{AIKN2021}]
\label{lem-energy2}
The function $\Gamma$ in the decomposition \eqref{eq-energy1} satisfies
\begin{equation}\label{eq-energy6}
\|\Gamma(t)\|_{H^{1}}
\lesssim 
\|\eta(t) \|_{H^{1}}
\end{equation} 
for all $t\in I_{\max}$. Moreover, we have 
\begin{equation}\label{eq-energy7}
\|\eta(t) \|_{H^{1}}
\sim
\|\eta(t) \|_{E}.
\end{equation}
\end{lemma}
As a summary, 
we obtain the following: 
\begin{proposition}[Proposition 4.0.4 of 
\cite{AIKN2021}]\label{prop-dist1} 
Let $\psi$ be a function in $H^{1}(\R^{3})$ satisfying 
$\mathcal{M}(\psi)= \mathcal{M}(Q_{\omega})$. 
We have the following the decomposition:  
\begin{align}
&\psi(t, x) = 
e^{i\theta(t)} (Q_{\omega} + 
\eta(t, x) ),
\quad 
\Omega(e^{-i\theta(t)} \psi, \partial_{\omega}Q_{\omega})\equiv 0,
\quad 
(e^{-i\theta(t)}\psi, \partial_{\omega}Q_{\omega})_{L_{\text{real}}^{2}}<0, 
\notag
\\[6pt] 
& 
\eta(t, x)
= 
\lambda_{+}(t) \mathcal{Y}_{\omega, +} + 
\lambda_{-}(t) \mathcal{Y}_{\omega, -} 
+\Gamma(t).
\label{eq-energy11}
\end{align}
Furthermore, 
There exists a constant $\delta_{E}(\omega) >0$ such that 
if  
$\|\eta(t)\|_{E}
\le 
4\delta_{E}(\omega)$, we obtain   
\begin{equation*}
\Bigm| 
\mathcal{E}(\psi)-\mathcal{E}(Q_{\omega})
+
\frac{e_{\omega}}{2} \big(\lambda_{+}(t) 
+ 
\lambda_{-}(t) \big)^{2}
-
\|\eta(t)\|_{E}^{2}
\Bigm| \le 
\frac{\|\eta(t)\|_{E}^{2}}{10}. 
\end{equation*}
\end{proposition}

\subsection{Distance function from the ground state}\label{sec-dist}
In this subsection, 
we introduce a distance function from the ground state $Q_{\omega}$ 
by using the linearized energy norm \eqref{eq-energy4}. 
For this, we 
fix a non-increasing smooth function $\chi$ on $[0,\infty)$ such that 
\begin{equation*}
\chi(r)=\left\{ \begin{array}{ccc}
1 &\mbox{if}& 0\le r\le 1,
\\ 
0 &\mbox{if}& r\ge 2. 
\end{array} \right.
\end{equation*}Then, we define a function 
$d_{\omega}\colon 
[0, T_{\max})
\to [0,\infty)$ by 
\begin{equation}\label{eq-dist3}
d_{\omega}^{2}(t)
:=
\|\eta(t)\|_{E}^{2}
+
\chi\Bigm( \frac{\|\eta(t)\|_{E}}{2\delta_{E}(\omega)}
\Bigm)
C_{\omega}(\psi(t)), 
\end{equation}
where $\delta_{E}(\omega)$ is the constant given by Proposition \ref{prop-dist1}, and 
\begin{equation*}
C_{\omega}(\psi(t))
:=\mathcal{E}(\psi) - 
\mathcal{E}(Q_{\omega})
+
\frac{e_{\omega}}{2} 
\big( \lambda_{+}(t) + 
\lambda_{-}(t) \big)^{2} -
\|\eta(t)\|_{E}^{2}.
\end{equation*}
Now, we introduce new parameters 
$\lambda_{1}(t)$ and $\lambda_{2}(t)$ defined by 
\begin{equation}\label{eq-dist6}
\lambda_{1}(t):=
\frac{\lambda_{+}(t)+\lambda_{-}(t)
}{2},
\qquad 
\lambda_{2}(t) 
:=\frac{\lambda_{+}(t) - 
\lambda_{-}(t)}{2}.
\end{equation}
It follows from \eqref{eq-mod5} and \eqref{eq-mod6} that 
\begin{align}
\label{eq-dist7}
\frac{d\lambda_{1}}{dt}(t)
&=
e_{\omega} \lambda_{2}(t)
+
( \Big\{ \frac{d\theta}{dt}(t)-\omega \Big\}\eta(t)
-N_{\omega}(\eta(t)), \, 
i \mathcal{Y}_{\omega, 2})_{L_{\text{real}}^{2}},
\\[6pt]
\label{eq-dist8}
\frac{d\lambda_{2}}{dt}(t)
&=
e_{\omega} \lambda_{1}(t)
-
( \Big\{ \frac{d\theta}{dt}(t)-\omega \Big\}\eta(t)
-N_{\omega}(\eta(t)), \, 
\mathcal{Y}_{\omega, 1} )_{L_{\text{real}}^{2}}.
\end{align}
We recall a property of the distance function $d_{\omega}(t)$:
\begin{lemma}[Lemma 4.0.5 of \cite{AIKN2021}]
\label{lem-dist1}
Assume that there exists an interval $I$ on which 
\begin{equation}
\label{eq-dist9}
\sup_{t\in I}d_{\omega}(t) 
\le \delta_{E}(\omega), 
\end{equation}
$\delta_{E}(\omega)$ is the constant given 
by Proposition \ref{prop-dist1}. 
Then, all of the following hold for all $t \in I$: 
\begin{align}
\label{eq-dist10}
&\frac{1}{2}\|\eta(t)\|_{E}^{2} \le d_{\omega}^{2}(t) 
\le \frac{3}{2}\|\eta(t)\|_{E}^{2},
\\[6pt]
&d_{\omega}^{2}(t)
= 
\mathcal{E}(\psi)-\mathcal{E}(Q_{\omega})
+
2e_{\omega} \lambda_{1}^{2}(t),
\label{eq-dist11}
\\[6pt]
\label{eq-dist12}
&\frac{d}{dt}d_{\omega}^{2}(t)
=
4e_{\omega}^{2} \lambda_{1}(t) \lambda_{2}(t) 
+ 4e_{\omega} \lambda_{1}(t) 
\big( \Big\{ \frac{d \theta}{d t}(t)-\omega \Big\}\eta(t) 
-N_{\omega}(\eta(t)), 
i \mathcal{Y}_{\omega, 2} \big)_{L_{\text{real}}^{2}}.
\end{align}
\end{lemma}


Next, we introduce a ``modified distance function'' $\widetilde{d}_{\omega}$. 

\begin{lemma}[Lemma 6.0.1 of 
\cite{AIKN2021}]\label{lem-mdist1}
Let $\omega_{2}$ be the frequency given 
by Theorem \ref{thm2-1}. 
Then, for any $\omega \in (0, \omega_{2})$, 
there exists $\gamma_{1}(\omega)>0$ with the following property: 
let $\psi$ be a solution to \eqref{nls} 
satisfying $\mathcal{M}(\psi)= \mathcal{M}(Q_{\omega})$. 
If ${\rm dist}_{H^{1}}
\big(\psi(t), \mathscr{O}(Q_{\omega})\big) 
\le \gamma_{1}(\omega)$, then 
\begin{equation}\label{eq-mdist1}
{\rm dist}_{H^{1}}\big(\psi(t), \mathscr{O}(Q_{\omega})\big)
\sim 
\|\eta(t)\|_{E}. 
\end{equation}
\end{lemma}
Then, we can obtain  
the following proposition from Lemma \ref{lem-mdist1}:
\begin{proposition}[Proposition 6.0.2 of 
\cite{AIKN2021}]\label{prop-mdist1}
Let $\omega_{2}$ be the frequency given 
by Theorem \ref{thm2-1}. Then, for any $\omega \in (0, \omega_{2})$, there exist a constant $\widetilde{\gamma}(\omega)\in (0, \delta_{E}(\omega))$~\footnote{
we recall that 
$\delta_{E}(\omega)$ denotes the constant given by Proposition \ref{prop-dist1}} and a continuous function $\widetilde{d}_{\omega} \colon [0, T_{\max}) \to [0,\infty)$ 
such that: 
\begin{equation*}
\widetilde{d}_{\omega}(t)
\sim 
{\rm dist}_{H^{1}}\big(\psi(t), \mathscr{O}(Q_{\omega})\big)
\end{equation*}
and if $\widetilde{d}_{\omega}(t)\le \widetilde{\gamma}(\omega)$, then 
\begin{equation}\label{eq-mdist8}
\widetilde{d}_{\omega}(t)= d_{\omega}(t), 
\end{equation}
where $d_{\omega}$ is the distance function defined by \eqref{eq-dist3}.
\end{proposition}

\subsection{One-pass theorem}\label{sec-opass1}
We are now in a position to state the one-pass theorem: 
\begin{theorem}[One-pass theorem (Theorem 7.0.1 
of \cite{AIKN2021})]\label{thm-opass1}
Let $\omega_{3}$ be the frequency given 
by Lemma \ref{tech-lem1}.
Then, for any $\omega \in (0,\omega_{3})$, 
there exists a positive constant 
$R_{*} > 0$ such that 
for any $R \in (0, R_{*})$ 
and any radial solution $\psi$ to \eqref{nls} 
with $\psi|_{t = 0} = \psi_{0}$
satisfying 
\begin{align}
\label{eq-opass8}
\mathcal{S}_{\omega}(\psi_0) = m_{\omega}, \qquad 
\mathcal{M}(\psi_0)=\mathcal{M}(Q_{\omega}),
\qquad 
\widetilde{d}_{\omega}(\psi_0)<R,
\end{align}
we have either 
\\[6pt]
{\rm (i)} $\widetilde{d}_{\omega}(\psi(t)) 
< R+R^{\frac{3}{2}}$ 
for all $t\in [0, T_{\max}^{+}]$; or 
\\[6pt]
{\rm (ii)} there exists $t_{*}>0$ such that $\widetilde{d}_{\omega}(\psi(t))\ge R+R^{\frac{3}{2}}$ for all $t\in [t_{*}, T_{\max}^{+}]$. 
\\
Here, $T_{\max}^{+}$ denotes the maximal existence time 
of $\psi$ in the positive direction. 
\end{theorem}
\begin{remark}
Actually, the result of Theorem 7.0.1 
of \cite{AIKN2021} is more general and we can 
treat the solutions satisfying $\mathcal{S}_{\omega}(\psi_0) < 
m_{\omega} + \varepsilon$ for sufficiently small $\varepsilon > 0$.
However, for simplicity of our presentation, we restrict ourselves 
to the threshold solutions. 
\end{remark}
\section{Analysis on $\mathcal{BA}_{\omega, -}$}
In this section, we study dynamics of 
solutions which start from $\mathcal{BA}_{\omega, -}$. 
\subsection{Convergence to the orbit of the ground state}
First, we shall show that 
if a solution with  
\eqref{nls} $\psi|_{t = 0} 
= \psi_{0} \in 
\mathcal{BA}_{\omega, -}$
exists globally in positive time direction, 
the solution converges to 
the orbit of the ground state 
exponentially as $t$ goes to infinity. 
More precisely, we obtain the following:
\begin{proposition}
\label{prop-conv-n}
Assume that $\omega \in 
(0, \omega_{3})$, 
where $\omega_{3} > 0$ is the 
constant given in Lemma 
\ref{tech-lem1}. 
Suppose that 
a solution $\psi$ to \eqref{nls} with $\psi|_{t = 0} = \psi_{0} \in 
\mathcal{BA}_{\omega, -}$ 
exists on  
$(- T_{\max}^{-}, \infty)$,  
Then, there exist constants 
$C>0$ 
and $c> 0$
such that 
\begin{equation} \label{expo-decay1}
\text{dist}_{H^{1}}(\psi(t), \mathcal{O}
(Q_{\omega})) 
\leq Ce^{- ct} 
\qquad \mbox{for all $t > 0$}. 
\end{equation}
\end{proposition}
We first show that 
the set $\mathcal{BA}
_{\omega, -}$ is invariant 
under the flow of \eqref{nls}. 
\begin{lemma}\label{inv-}
If $\psi_{0} \in \mathcal{BA}_{\omega, -}$,  
we have $\psi(t) \in \mathcal{BA}_{\omega, -}$ for all $t \in I_{\max}$, 
where $\psi$ is the solution to \eqref{nls} with $\psi|_{t = 0} = \psi_{0}$. 
\end{lemma}
\begin{proof}
It follows from the conservation laws 
\eqref{eq-con} that $\mathcal{S}_{\omega} 
(\psi(t)) = m_{\omega}$. 
We shall show 
$\mathcal{K}(\psi(t)) < 0$ for all 
$t \in I_{\max}$ by contradiction. 
Suppose to the contrary that 
there exist $\omega_{*} \in 
(0, \omega_{3}), \psi_{0, *} \in 
\mathcal{BA}_{\omega, -}$ 
and $t_{*} \in I_{\max}$ 
such that 
$\mathcal{K}(\psi(t_{*})) \geq 0$. 
Then, from the continuity of 
$\mathcal{K}(\psi(t))$ and 
$\mathcal{K}(\psi_{0, *}) < 0$, there exists $t_{1} \in (0, t_{*}]$ such that 
$\mathcal{K}(\psi(t_{1})) = 0$. 
Thus, $\psi(t_{1})$ satisfies 
$\mathcal{S}_{\omega_{*}}(\psi(t_{1})) = m_{\omega_{*}}$ and 
$\mathcal{K}(\psi(t_{1})) = 0$. 
Since $\psi(t_{1})$ is a minimizer 
of $m_{\omega_{*}}$ (ground state of 
$m_{\omega_{*}}$), we have that 
$\psi(t) = e^{i \omega_{*} t + \theta} Q_{\omega_{*}}$ for some 
$\theta \in \R$, 
which contradicts $
\mathcal{K} (\psi(0)) = 
\mathcal{K} (\psi_{0, *}) < 0$. 
\end{proof}
Next, we see that the following convergence 
result holds. 
\begin{proposition}\label{conv-gs}
Let $\{u_{n}\}$ be a sequence 
in $H^{1}_{\text{rad}}(\R^{3})$ 
satisfying 
$\lim_{n \to 
\infty} 
\mathcal{S}_{\omega}(u_{n}) 
= m_{\omega}$ and 
$\lim_{n \to \infty}
\mathcal{K}(u_{n}) = 0$. 
Then, we have 
$\lim_{n \to \infty} 
\text{dist}_{H^{1}} 
(u_{n}, \mathcal{O}(Q_{\omega})) 
= 0$. 
\end{proposition}
We can prove Proposition \ref{conv-gs} 
by a standard argument. 
However, for the sake of the completeness, 
we shall give the proof in Appendix \ref{conv} below. 
Next, we recall several estimates, 
which are needed later.
\begin{lemma}\label{lem-estR}
Let $\theta(t)$ and $\eta$ be the functions given in 
\eqref{eq-dec1}, 
$N_{\omega}(\eta)$ be the 
nonlinear function defined by \eqref{eq-dec8},  
$\mathcal{Y}_{\omega, \pm}$ 
be the eigenfunctions of $- i \mathcal{L}_{\omega}$, 
$\mathcal{Y}_{\omega, i}\; (i = 1, 2)$ 
be the function defined by \eqref{eq-egf1} 
and $\lambda_{1}(t)$ the parameter given in 
\eqref{eq-dist6}. 
There exists sufficiently small $\delta_{X} > 0$ and 
$C_{1}>0$ 
such that 
as long as the solution $\psi$ to 
\eqref{nls} 
satisfies $d_{\omega}(\psi(t)) < \delta_{X}$, 
we have
\begin{equation}\label{eq-ejec24}
\bigg|\Omega\big( \Big\{ \frac{d \theta}{d t}(t)-\omega \Big\}
\eta(t) -N_{\omega}(\eta(t)), 
\mathcal{Y}_{\omega, i} \big) 
\bigg|
\leq C_{1}
\big|\lambda_{1}(t)\big|^{2} 
\qquad \mbox{for $i = 1, 2$},
\end{equation} 

\begin{equation}\label{eq-ejec25}
\bigg|\big( \Big\{ \frac{d \theta}{d t}(t)-\omega \Big\}
\eta(t) -N_{\omega}(\eta(t)), \mathcal{Y}_{\omega, \pm} \big)_{L_{\text{real}}^{2}} 
\bigg| \leq C_{1}
\big|\lambda_{1}(t)\big|^{2}. 
\end{equation}
\end{lemma}
See \cite[(5.30) and (5.31)]{AIKN2021} for 
the proof of Lemma \ref{lem-estR}. 
We are now in a position to prove 
Proposition \ref{prop-conv-n}. 
\begin{proof}[Proof of Proposition \ref{prop-conv-n}
]
We divide the proof into four 
steps. 

\textbf{(Step 1). }
We claim that 
there exists a sequence $\{t_{n}\}$ in $(0, \infty)$ with 
$\lim_{n \to \infty} t_{n} = \infty$
such that 
\begin{equation}\label{lem-sp1}
\lim_{n \to \infty} 
\text{dist}_{H^{1}}(\psi(t_{n}), 
\mathcal{O}(Q_{\omega}))
= 0. 
\end{equation}
Since $\mathcal{BA}_{\omega, -}$ is an 
invariant set (see Lemma \ref{inv-}), 
we see 
that $\mathcal{S}_{\omega}(\psi(t)) = m_{\omega}$ and 
$\mathcal{K}(\psi(t)) < 0$ for all $t \in (- T_{\max}^{-}, \infty)$. 
Then, we obtain $\limsup_{t \to \infty} \mathcal{K}(\psi(t)) = 0$.  
Otherwise, the solution $\psi(t)$ 
blows up in finite time (see the proof of 
Theorem 1.3 in \cite{AIKN2012}). 
Thus, there exists a sequence $\{t_{n}\}$ in $[0, \infty)$ 
with $\lim_{n \to \infty} t_{n} = \infty$
such that $\lim_{n \to \infty} 
\mathcal{K}(\psi(t_{n})) = 0$. 
Then, up to a subsequence, 
we see from Proposition \ref{conv-gs} that $\lim_{n \to \infty} 
\text{dist}_{H^{1}} 
(\psi(t_{n}), 
\mathcal{O}(Q_{\omega})) 
= 0$. 
Thus, \eqref{lem-sp1} holds. 

\textbf{(Step 2). } 
\eqref{lem-sp1} together with 
one-pass theorem 
(Theorem \ref{thm-opass1}) 
yields that 
for any $R > 0$, there exists 
$N \in \mathbb{N}$ such that 
$\widetilde{d}_{\omega} (\psi(t)) < R + R^{\frac{3}{2}}$ 
for all $t \in [t_{N}, \infty)$. 
Since $R > 0$ is arbitrary, 
we have 
$\lim_{t \to \infty} \widetilde{d}_{\omega} (\psi(t)) = 0$. 
Then, it follows from \eqref{eq-mdist8} that  
\begin{equation*}
d_{\omega} (\psi(t))
< 2 R \qquad \mbox{for all 
$t \in [t_{N}, \infty)$}. 
\end{equation*} 
In addition, by \eqref{eq-energy0}, 
\eqref{eq-energy4} and 
\eqref{eq-dist10}, we have
\begin{equation} \label{eq-bu2}
\lambda_{+}^{2}(\psi(t)) 
+ \lambda_{-}^{2}(\psi(t)) \leq 
\frac{2}{e_{\omega}} 
\|\eta(t)\|_{E}^{2} \lesssim R 
\qquad \mbox{for all $t \in 
[t_{N}, \infty)$}
\end{equation}
Therefore, 
letting 
\begin{equation}\label{eq-ejec18-2}
v_{0} 
:= \max_{t \geq 0} \{ 
|\lambda_{1}(t)|\} > 0.
\end{equation}
we may assume that  
$v_{0} \in (0, \widetilde{\gamma}_{\omega})$ 
by replacing 
$t$ by $t + T \; (T \gg 1)$. 

\textbf{(Step 3). 
}
Using an idea of Nakanishi and Schlag~\cite{NS2012}, 
we shall show that 
$|\lambda_{1}(t)|$ is a decreasing 
function of $t$. 
Suppose to the contrary that 
there exists $t_{0} > 0$ such that 
$|\lambda_{1}(t_{0})| 
= \max_{t \geq 0}{|\lambda_{1}(t)|} = v_{0}$. 
Observe from $\psi_{0} \in 
\mathcal{BA}_{\omega, -}$ that 
$\mathcal{E}(\psi) = 
\mathcal{E}(Q_{\omega})$. 
Then, 
it follows from \eqref{eq-dist11} 
that 
  \begin{equation} \label{eq-dist14}
d_{\omega}^{2}(t) = 2 e_{\omega} \lambda_{1}^{2}(t).
  \end{equation}
This yields that 
$d_{\omega}^{2}(t_{0}) 
= \max_{t \geq 0} d_{\omega}^{2}(t)$. 
Then, one has 
\begin{equation*}
0 =
\frac{d}{dt} d_{\omega}^{2}
\big(\psi(t)\big)\bigg|_{t=t_{0}}. 
\end{equation*}
This together with \eqref{eq-dist12} 
yields 
\begin{equation}\label{eq-ejec27-2}
0 = 
e_{\omega}^{2} 
\lambda_{1}(t_{0}) 
\lambda_{2}(t_{0}) 
+
e_{\omega} \lambda_{1}(t_{0}) 
\big
( \Big\{ \frac{d \theta}{d t}(t_{0}) - 
\omega \Big\}\eta(t_{0}) -N_{\omega}(\eta(t_{0})), 
i \mathcal{Y}_{\omega, 2} 
\big)_{L^{2}_{\text{real}}}.
\end{equation}
Combining \eqref{eq-ejec27-2} with \eqref{eq-ejec24}, we obtain that 
\begin{equation}\label{eq-ejec28-2}
0 \leq  
e_{\omega}^{2} \, {\rm sgn}
[\lambda_{1}(t_{0})]|\lambda_{1}(t_{0})| \lambda_{2}(t_{0})
+ Ce_{\omega} |\lambda_{1}(t_{0})|^{3} 
\end{equation} 
for some constant $C>0$. 
It follows from \eqref{eq-ejec28-2} that 
\begin{equation}\label{eq-ejec29-2}
- C |\lambda_{1}(t_{0})|^{2} 
\le 
e_{\omega} {\rm sgn}[\lambda_{1}(t_{0})] 
\lambda_{2}(t_{0}). 
\end{equation}
Suppose that ${\rm sgn}[\lambda_{1}(t_{0})]=1$. 
Then, since $|\lambda_{1}(t_{0})| \ll 1$, 
we see from
\eqref{eq-ejec18-2}, \eqref{eq-ejec29-2} 
and \eqref{eq-dist6} that 
\begin{equation} \label{eq-bu3}
0<\lambda_{1}(t_{0}) 
\leq 2 \lambda_{1}(t_{0}) 
- \frac{2 C}{e_{\omega}}
\big| \lambda_{1}(t_{0})
\big|^{2}
\leq 2 
(\lambda_{1}(t_{0}) 
+ \lambda_{2}(t_{0})) 
= 2 \lambda_{+}(t_{0}).
\end{equation} 
Suppose next that 
${\rm sgn}[\lambda_{1}
(t_{0})] = -1$. 
Then, \eqref{eq-ejec29-2} becomes 
$e_{\omega} \lambda_{2}(t_{0}) 
\le C |\lambda_{1}(t_{0})|^{2}$. 
Since $|\lambda_{1}(t_0)| = R \ll 1$, 
we see from \eqref{eq-dist6} that 
\begin{equation} \label{eq-bu4}
\lambda_{+}(t_{0})
=
\lambda_{1}(t_{0})+
\lambda_{2}(t_{0})
\le 
\lambda_{1}(t_{0}) + 
\frac{C}{e_{\omega}}|\lambda_{1}(t_{0})|^{2} 
\leq \frac{\lambda_{1}(t_{0})}{2} < 0. 
\end{equation} 
Thus, we conclude from \eqref{eq-bu3} 
and \eqref{eq-bu4} that 
\begin{equation}\label{eq-ejec30-2}
|\lambda_{1}(t_{0})| 
\leq 2|\lambda_{+} (t_{0})|.
\end{equation}
From \eqref{eq-mod5}, 
\eqref{eq-ejec24}, \eqref{eq-ejec30-2} and 
$|\lambda_{1}(t_{0})| = \max_{t \geq 0}{|\lambda_{1}(t)|}$,  we have 
\begin{equation*}
\begin{split}
|\lambda_{+}(t)| 
& \geq e^{e_{\omega} (t - t_{0})}
|\lambda_{+}(t_{0})| 
- C_{1} |\lambda_{1}(t_{0})|^{2}
\int_{t_{0}}^{t} 
e^{e_{\omega}(t - s)} ds \\
& \geq e^{e_{\omega} (t - t_{0})}
|\lambda_{+}(t_{0})| 
- \frac{C_{1}}{e_{\omega}} 
|\lambda_{1}(t_{0})|^{2} 
e^{e_{\omega} (t - t_{0})} \\
& \geq 
e^{e_{\omega} (t - t_{0})}
\left(1
- \frac{4 C_{1} 
|\lambda_{+}(t_{0})|}{e_{\omega}} 
\right)|\lambda_{+}(t_{0})| 
\geq \frac{e^{e_{\omega} 
(t - t_{0})}}{2} 
|\lambda_{+}(t_{0})|
\to \infty 
\qquad \mbox{as $t \to \infty$}, 
\end{split}
\end{equation*}
which contradicts \eqref{eq-bu2}. 
We have used the fact that 
$|\lambda_{+}(t_{0})| \ll 1$ in 
the last inequality (see \eqref{eq-bu2}). 
Thus, we find that $|\lambda_{1}(t)|$ 
is a decreasing function of $t$. 

\textbf{(Step 4). 
}
We shall show \eqref{expo-decay1}. 
We first consider the case that 
$\lambda_{1}(t)$ is a 
decreasing function of $t>0$. 
Since $\lim_{t \to \infty} 
\lambda_{1}(t) = 0$, we have
$\lambda_{1}(t) > 0$. 

It follows from \eqref{eq-dist8}, 
\eqref{eq-ejec24} 
and $\max_{t \geq 0} 
|\lambda_{1}(t)| = R \ll 1$ that 
\[
\frac{d \lambda_{2}(t)}{d t} 
\geq e_{\omega} \lambda_{1}(t) 
- C_{1} 
|\lambda_{1}(t)|^{2} 
\geq \frac{e_{\omega}}{2} 
\lambda_{1}(t) > 0
\]
This together with $\lim_{t \to \infty} 
|\lambda_{2}(t)| = 0$
implies that 
$\lambda_{2}(t) < 0$ for all $t>0$. 
Since 
$\lambda_{1}(t) > 0$ and 
$\lambda_{2}(t) < 0$ for all $t > 0$, 
we see from \eqref{eq-dist6} 
that 
$
- \lambda_{-}(t) < 
\lambda_{+}(t) < \lambda_{-}(t)$ 
for all $t > 0$. 
Therefore, we have 
$\lambda_{-}(t) > 0$ and 
\begin{equation}\label{eq-ode1}
|\lambda_{+}(t)| 
\leq |\lambda_{-}(t)|.
\end{equation} 
From \eqref{eq-mod6}, 
\eqref{eq-ejec25},  
\eqref{eq-ode1} and 
$|\lambda_{-}(t)| \ll 1$, 
we find that 
\[
\frac{d \lambda_{-}}{d t}(t) 
\leq 
- e_{\omega} \lambda_{-}(t) 
+ 4 C_{1} |\lambda_{-}(t)|^{2} 
\leq 
- \frac{e_{\omega}}{2} 
\lambda_{-}(t). 
\]
It follows that 
$0 < \lambda_{-}(t) \leq 
\lambda_{-}(0) e^{- 
\frac{e_{\omega}}{2}t}$. 
We see from 
\eqref{eq-ode1}
that $|\lambda_{+}(t)| 
< \lambda_{-}(0) e^{- 
\frac{e_{\omega}}{2}t}$. 
It follows from \eqref{eq-dist6} that 
$|\lambda_{1}(t)| \leq \lambda_{-}(0) e^{- 
\frac{e_{\omega}}{2}t}$. 
This together with Proposition \ref{prop-mdist1}, 
\eqref{eq-dist14} implies that \eqref{expo-decay1} holds.

We can prove the case where 
$\lambda_{1}(t)$ is an increasing 
function of $t > 0$ similarly. 
Thus, we omit it. 
\end{proof}
\subsection{Blowup in negative time direction}
Next, we shall show that 
if a solution starts from 
$\mathcal{BA}_{\omega, -}$ exists
globally in positive time direction, 
the solution must blow up in finite negative time: 
\begin{proposition}\label{K-blowup}
Assume that $\omega \in 
(0, \omega_{3})$, 
where $\omega_{3} > 0$ is the 
constant given in Lemma 
\ref{tech-lem1}. 
Suppose that a solution $\psi$ to \eqref{nls} with $\psi|_{t = 0} 
= \psi_{0} \in 
\mathcal{BA}_{\omega, -}$ exists on 
$(-T_{\max}^{-}, \infty)$. 
Then, 
$\psi$ 
blows up in finite negative time, 
that is, $T_{\max}^{-} < \infty$. 
\end{proposition}
To prove Proposition \ref{K-blowup}, 
we first 
consider a solution which has 
a finite variance. 
Namely, we shall show the following: 
\begin{lemma}\label{U-blowup}
Assume that $\omega \in 
(0, \omega_{3})$, 
where $\omega_{3} > 0$ be the 
constant given in Lemma 
\ref{tech-lem1}.
Suppose that a solution $\psi$ to \eqref{nls} with $\psi|_{t = 0} 
= \psi_{0} \in \mathcal{BA}_{\omega, -}$ defined on 
$(-T_{\max}^{-}, \infty)$ satisfies
\begin{equation*}
|x| \psi \in L^{2}(\R^{3}). 
\end{equation*}
Then, 
$\psi$ 
blows up in finite negative time, 
that is, $T_{\max}^{-} < \infty$. 
\end{lemma}
We can prove Lemma 
\ref{U-blowup} by a similar argument in 
the proof of Proposition 5.1 (Page 25)
of \cite{DR2010}. 
Thus, we omit the proof.

\begin{lemma}\label{lem-blowup1} 
Assume that $\omega \in 
(0, \omega_{3})$, 
where $\omega_{3} > 0$ is the 
constant given in Lemma 
\ref{tech-lem1}.
Suppose that 
a solution $\psi$ to \eqref{nls} with $\psi|_{t = 0} 
= \psi_{0} \in 
\mathcal{BA}_{\omega, -}$ is defined on 
$(-T_{\max}^{-}, \infty)$ and 
$T_{\max}^{-} = \infty$. 
Then, $\lim_{t \to -\infty} \text{dist}_{H^{1}} 
(\psi(t), \mathcal{O}(Q_{\omega})) = 0$ and 
there exists a constant 
$C_{1} > 0$ such that
\begin{equation}\label{eq-ode13}
\mathcal{K}(\psi(t)) 
< - C_{1} \|\eta(t)\|_{H^{1}} 
\qquad 
\mbox{for all 
$t \in (-\infty, 0)$}. 
\end{equation}
\end{lemma}
\begin{proof}
We divide the proof into 
four steps. 

\textbf{(Step 1). } 
Suppose that there exists $\delta > 0$ 
such that $\mathcal{K}(\psi(t)) 
< - \delta$ for all $t \in 
(- T_{\max}^{-}, \infty)$. 
Then, we see that $T_{\max}^{-} 
< \infty$. 
Thus, it suffices to consider the case of  
$\limsup_{t \to -\infty} 
K(\psi(t)) = 0$. 
Then, there exists a sequence 
$\{t_{n}\}$ in $(- \infty, 0)$ 
with $\lim_{n \to \infty} t_{n} 
= -\infty$ such that $\lim_{n \to \infty} 
K(\psi(t_{n})) = 0$. 
Then, by a similar way to \textbf{(Step 1)} 
--\textbf{(Step 3)} in the proof of Proposition \ref{prop-conv-n}, 
we see that 
for any $R > 0$, there exists $t_{N}$ in $(-\infty, 0)$ 
such that  
\begin{equation*}
d_{\omega} (\psi(t))
< C R \qquad \mbox{for all 
$t \in (-\infty, t_{N}]$},  
\end{equation*} 
\begin{equation} \label{eq-bu2-1}
\lambda_{+}^{2}(\psi(t)) 
+ \lambda_{-}^{2}(\psi(t)) \leq 
\|\eta(t)\|_{E}^{2} \lesssim R 
\qquad \mbox{for all $t \in 
(-\infty, t_{N}]$}. 
\end{equation}
Since $R > 0$ is arbitrary, we have 
$\lim_{t \to -\infty} \text{dist}_{H^{1}} 
(\psi(t), \mathcal{O}(Q_{\omega})) = 0$. 
Replacing 
$t$ by $t - T \; (T \gg 1)$, 
we may assume that for any 
$v_{0} \in (0, \widetilde{\gamma}_{\omega})$,  
\begin{equation}\label{eq-ejec18-2-1}
v_{0} 
= \max_{t \leq 0} \{ 
|\lambda_{1}(t)|\} > 0.
\end{equation}   
In addition, we see that 
\begin{equation} \label{eq-bu11}
\mbox{
$|\lambda_{1}(t)|$ is an increasing 
function of $t$. 
}
\end{equation}
   
\textbf{(Step 2). 
}
We claim that 
\begin{equation}\label{eq-ode4}
\lim_{t \to - \infty} 
\biggl|\frac{\lambda_{-}(t)}
{\lambda_{+}(t)} \biggl| = 0. 
\end{equation} 
Suppose to the contrary that 
there exists a constant 
$\delta > 0$ such that 
\begin{equation}\label{eq-ode5}
\limsup_{t \to - \infty} 
\biggl|\frac{\lambda_{-}(t)}
{\lambda_{+}(t)} \biggl| \geq \delta. 
\end{equation} 
Then there exists a sequence 
$\{t_{n}\}$ in $(-\infty, 0)$ 
with $\lim_{n \to \infty} t_{n} 
= -\infty$ such that 
\[
\biggl|\frac{\lambda_{-}(t_{n})}
{\lambda_{+}(t_{n})} \biggl| 
\geq \frac{\delta}{2}.
\]
This together with \eqref{eq-dist6} yields that
\begin{equation} \label{eq-ode6}
|\lambda_{1}(t_{n})| 
\leq |\lambda_{+}(t_{n})|
+ |\lambda_{-}(t_{n})| 
\leq \frac{2 + \delta}{\delta} 
|\lambda_{-}(t_{n})| 
\end{equation}
From \eqref{eq-mod6}, 
\eqref{eq-ejec24}, \eqref{eq-bu11} and 
\eqref{eq-ode6}, we have 
\begin{equation*}
\begin{split}
|\lambda_{-}(t)| 
& \geq 
e^{- e_{\omega} (t - t_{n})}
|\lambda_{-}(t_{n})| 
- C_{1} |\lambda_{1}(t_{n})|^{2}
\int_{t}^{t_{n}} 
e^{- e_{\omega}(t - s)} ds \\
& \geq e^{- e_{\omega} 
(t - t_{n})}
|\lambda_{-}(t_{n})| 
- \frac{C_{1}}{e_{\omega}} 
|\lambda_{1}(t_{n})|^{2} 
e^{- e_{\omega} (t - t_{n})} \\
& \geq 
\left(1 - \frac{(2 + \delta)
C_{1}}{\delta e_{\omega}} 
|\lambda_{1}(t_{n})|
\right) e^{- e_{\omega} 
(t - t_{n})}
|\lambda_{-}(t_{n})|. 
\end{split}
\end{equation*}
Since $\lim_{t \to - \infty} 
|\lambda_{1}(t)| = 0$, 
there exists 
a sufficiently large 
$n_{0} \in \mathbb{N}$
such that 
\[
1 - \frac{(2 + \delta)
C_{1}}{\delta e_{\omega}} 
|\lambda_{1}(t_{n_0})| 
\geq \frac{1}{2}. 
\]
Then, we obtain 
\[
|\lambda_{-}(t)| 
\geq \frac{e^{- e_{\omega}(t - t_{n_0})}}{2}
|\lambda_{-}(t_{n_0})| \to \infty 
\qquad \mbox{as $t \to -\infty$}, 
\]
which is a contradiction. 
Thus, \eqref{eq-ode4} holds. 

\textbf{(Step 3).
}
Now, we see from \eqref{eq-uni2-2}, 
\eqref{eq-energy11} and \eqref{eq-egf8} that 
\begin{equation}\label{eq-egf11}
(Q_{\omega}, \Gamma(t))_{L_{\text{real}}^{2}}
= - \mathcal{M}(\eta(t)).
\end{equation}
Moreover, we see from \eqref{eq-dec4} and \eqref{eq-dec5} that 
\begin{equation}\label{eq-egf12}
\begin{split}
\mathcal{K}^{\prime}
(Q_{\omega}) + 2\omega 
Q_{\omega}
& = 
- 2\Delta Q_{\omega} 
+ 2\omega Q_{\omega}
- 3 Q_{\omega}^{3}
- 6 Q_{\omega}^{5} \\
&=
\frac{3}{2}L_{\omega,-} Q_{\omega}
+
\frac{1}{2} L_{\omega,+}Q_{\omega} 
- 2Q_{\omega}^{5}
\\[6pt]
&=
\frac{1}{2} 
L_{\omega,+}Q_{\omega}
- 2 Q_{\omega}^{5}.
\end{split}
\end{equation}
We note that the decomposition \eqref{eq-energy11} of $\eta(t)$ 
is expressed as follows in terms of the functions 
$\mathcal{Y}_{\omega, 1}$ and $\mathcal{Y}_{\omega, 2}$:
\begin{equation}\label{eq-egf2}
\eta(t)
=
2 \lambda_{1}(t)\mathcal{Y}_{\omega, 1}+ 2 i \lambda_{2}(t)\mathcal{Y}_{\omega, 2}+\Gamma(t).
\end{equation}
This together with 
\eqref{eq-egf11}, 
\eqref{eq-egf12}, 
\eqref{eq-egf2} and 
\eqref{eq-uni2-2} shows that 
\begin{equation*}
\begin{split}
&\langle \mathcal{K}^{\prime}
(Q_{\omega}), \eta(t) 
\rangle_{H^{-1}, H^{1}} 
\\[6pt]
&= 
\langle 
\mathcal{K}^{\prime}
(Q_{\omega}) 
+ 2\omega Q_{\omega}, \eta(t) 
\rangle_{H^{-1}, H^{1}}
- 2\omega (Q_{\omega},\eta(t))_{L_{\text{real}}^{2}}
\\[6pt]
& = \langle 
\frac{1}{2} L_{\omega, +} Q_{\omega}
- 2 Q_{\omega}^{5}, 
2 \lambda_{1}(t)\mathcal{Y}_{\omega, 1}+ 2 i \lambda_{2}(t)\mathcal{Y}_{\omega, 2}+\Gamma(t) 
\rangle_{H^{-1}, H^{1}} 
- 2\omega (Q_{\omega}, 
\eta(t))_{L_{\text{real}}^{2}}
\\[6pt]
& = \langle 
\frac{1}{2} L_{\omega, +} Q_{\omega}
- 2 Q_{\omega}^{5}, 
2 \lambda_{1}(t)\mathcal{Y}_{\omega, 1} + \Gamma(t) 
\rangle_{H^{-1}, H^{1}} 
+ 2\omega \mathcal{M}(\eta(t)). 
\end{split}
\end{equation*}
Here, we have used the fact that 
$Q_{\omega}, L_{\omega, +} Q_{\omega}$ 
and $\mathcal{Y}_{\omega, 2}$ are 
real-valued functions in the last equality. 
Observe from \eqref{eq-dec4} that 
$\frac{1}{2} L_{\omega, +} Q_{\omega}
- 2 Q_{\omega}^{5} 
= - Q_{\omega}^{3} - 4 Q_{\omega}^{5}$. 
This together with \eqref{eq-pre1} 
yields that  
\begin{equation}\label{eq-egf13}
\begin{split}
\langle \mathcal{K}^{\prime}
(Q_{\omega}), \eta(t) 
\rangle_{H^{-1}, H^{1}} 
& 
= \lambda_{1}(t)
(L_{\omega,+} Q_{\omega}, \mathcal{Y}_{\omega, 1} 
)_{L^{2}}
- 4 \lambda_{1}(t)
(Q_{\omega}^{5}, 
\mathcal{Y}_{\omega, 1} 
)_{L^{2}}
\\[6pt]
&\quad - (Q_{\omega}^{3}
+ 4Q_{\omega}^{5}, \Gamma(t) 
)_{L_{\text{real}}^{2}} + 2 \omega \mathcal{M}(\eta(t)) 
\\[6pt]
&=
- e_{\omega} 
\lambda_{1}(t)( Q_{\omega}, \mathcal{Y}_{\omega, 2})_{L^{2}}
- 4 \lambda_{1}(t)
(Q_{\omega}^{5}, 
\mathcal{Y}_{\omega, 1} 
)_{L^{2}}
\\[6pt]
&\quad 
- (Q_{\omega}^{3}
+ 4 Q_{\omega}^{5}, \Gamma(t) 
)_{L_{\text{real}}^{2}} + 2 \omega \mathcal{M}(\eta(t)).
\end{split}
\end{equation}
Taylor's expansion of $\mathcal{K}$ around $Q_{\omega}$ together with $\mathcal{K}(Q_{\omega})=0$ and \eqref{eq-egf13} shows that 
\begin{equation}\label{eq-egf10-1}
\begin{split}
\mathcal{K}(\psi(t))
&=
\mathcal{K}(Q_{\omega}+\eta(t))
=
\langle \mathcal{K}'(Q_{\omega}), \eta(t)\rangle 
+O(\|\eta(t)\|_{H^{1}}^{2})
\\[6pt]
&=
- e_{\omega} 
\lambda_{1}(t)( Q_{\omega}, \mathcal{Y}_{\omega, 2})_{L^{2}}
-
4 \lambda_{1}(t)
(Q_{\omega}^{5}, 
\mathcal{Y}_{\omega, 1} )_{L^{2}}
- (Q_{\omega}^{3}, \Gamma(t))_{L_{\text{real}}^{2}}
\\[6pt]
&\quad 
- 4 (Q_{\omega}^{5},\Gamma(t))_{L_{\text{real}}^{2}} 
+
O(\| \eta(t) \|_{H^{1}}^{2}).
\end{split}
\end{equation}

\textbf{(Step 4). 
}
From \eqref{eq-ode4}, 
for any $\varepsilon>0$, 
there exists $T_{\varepsilon} > 0$ 
such that 
\begin{equation}\label{eq-ode7}
|\lambda_{-}(t)| < \varepsilon 
|\lambda_{+}(t)| 
\qquad \mbox{for all 
$t < - T_{\varepsilon}$}. 
\end{equation}
It follows from 
\eqref{eq-energy3}, the condition $\mathcal{E}(\psi) 
= \mathcal{E}(Q_{\omega})$,     
\eqref{eq-dist10}, \eqref{eq-dist11} and \eqref{eq-energy0} that 
\[
e_{\omega} |\lambda_{+}(t)||\lambda_{-}(t)| + C |\lambda_{1}(t)|^{3} 
\geq \|\Gamma(t)\|_{H^{1}}^{2}.
\]
This together with 
\eqref{eq-ode7} and 
\eqref{eq-dist6} yields that 
\begin{equation}\label{eq-ode8}
\|\Gamma(t)\|_{H^{1}}^{2} 
\leq e_{\omega} \varepsilon
|\lambda_{+}(t)|^{2} + C |\lambda_{1}(t)|^{3}  
\leq 
C \varepsilon
|\lambda_{1}(t)|^{2}
\end{equation}
Suppose that there exists 
$t_{*} < - T_{\varepsilon}$ 
such that $\lambda_{1}(t_{*}) > 0$. 
Then, it follows from 
\eqref{eq-egf10-1}, 
\eqref{eq-egf10},
\eqref{eq-ejec4} and 
\eqref{eq-ode8} that 
\begin{equation} \label{eq-ode9}
\begin{split}
0 > K(\psi(t_{*})) 
& \geq 
\frac{e_{\omega}}{2}
\lambda_{1}(t_{*}) 
|(Q_{\omega}, 
\mathcal{Y}_{\omega, 
2})_{L^{2}}| - 
C(\|Q_{\omega}\|_{L^{\infty}}^{2} 
+ \|Q_{\omega}\|_{L^{\infty}}^{4})\|Q_{\omega}\|_{L^{2}}
\|\Gamma(t_{*})\|_{L^{2}} \\
& \geq \frac{e_{\omega}}{4}
\lambda_{1}(t_{*}) 
|(Q_{\omega}, 
\mathcal{Y}_{\omega, 
2})_{L^{2}}| \geq 0, 
\end{split}
\end{equation}
which is a contradiction. 
Thus, we see from \eqref{eq-bu11} and $\lim_{t \to -\infty} 
\lambda_{1}(t) = 0$ that 
$\lambda_{1}(t) 
< 0$ for $t < - T_{\varepsilon}$. 
Then, from 
\eqref{eq-dist14}, 
\eqref{eq-dist10} and 
\eqref{eq-energy7}, 
we have $\lambda_{1}(t) 
\leq - C \|\eta(t)\|_{H^{1}}$. 
Then, by a similar argument as in 
\eqref{eq-ode9}, we obtain 
\[
\mathcal{K}(\psi(t)) 
\leq \frac{e_{\omega}}{4}
\lambda_{1}(t) 
|(Q_{\omega}, 
\mathcal{Y}_{\omega, 
2})_{L^{2}}| 
\leq - \frac{e_{\omega}}{4}
C |(Q_{\omega}, 
\mathcal{Y}_{\omega, 
2})_{L^{2}}| 
\|\eta(t)\|_{H^{1}} 
\qquad \mbox{for all 
$t < - T_{\varepsilon}$}. 
\]
Therefore, 
from the continuity of 
$\mathcal{K}(\psi(t))$ and 
$\eta(t)$, we see that 
\eqref{eq-ode13} holds. 
\end{proof}

Let $\varphi $ be a radial 
function such that 
\begin{equation*}
\varphi \in C^{\infty}(\R^{3}), \qquad 
\varphi(r) \geq 0, \quad \varphi^{\prime \prime}(r) \leq 2 
\qquad (r \geq 0), \qquad 
\varphi(r)
= 
\begin{cases}
r^{2} & \qquad (0 \leq r \leq 1), \\
0 & \qquad (r \geq 2).
\end{cases}
\end{equation*}
For $R > 0$, we put 
\begin{equation} \label{eq-ode15}
y_{R}(t) := \int_{\R^{3}} R^{2} \varphi(\frac{x}{R}) |\psi(t, x)|^{2} dx. 
\end{equation}
Then, we obtain 
\begin{equation} \label{eq-ode11}
y_{R}^{\prime \prime}(t) 
= 8 \mathcal{K}(\psi) + A_{R}(\psi(t)), 
\end{equation}
where 
\begin{equation*}
\begin{split}
A_{R}(\psi(t)) 
& := 4\int_{\R^{3}} 
\left( \varphi^{\prime \prime}(\frac{x}{R}) - 2 
\right)|\nabla \psi(t, x)|^{2} dx
- \int_{\R^{3}} 
\left(\Delta \varphi(\frac{x}{R}) - 6
\right) 
|\psi(t, x)|^{4} dx \\
& \quad 
- \frac{4}{3} 
\int_{\R^{3}}
\left( 
\Delta \varphi(\frac{x}{R}) - 6
\right) |\psi(t, x)|^{6} dx 
+ \frac{1}{R^{2}} 
\int_{\R^{3}} \Delta^{2} \varphi(\frac{x}{R}) |\psi(t, x)|^{2} dx.
\end{split}
\end{equation*}
Then, by a similar argument in 
\cite[Claim 4.3]{DM2009} 
(see also \cite[Section 5.2]{DR2010}), 
we can prove the following: 
\begin{lemma} \label{lem-blowup2}
For any $\varepsilon > 0$, 
there exists $R_{\varepsilon} > 0$ 
such that 
for any $R \geq R_{\varepsilon}$, 
we have 
\begin{equation} \label{eq-ode10}
|A_{R}(\psi(t))| \leq 
\varepsilon \|\eta(t)\|_{H^{1}}. 
\end{equation}
\end{lemma}

We are now in a position to prove 
Proposition \ref{K-blowup}. 
Actually, the proof is similar to 
that of \cite{DR}. 
However, for the reader's convenience, 
we shall give the proof. 
\begin{proof}[Proof of Proposition 
\ref{K-blowup}]
Suppose to the contrary 
that $T_{\max}^{-} 
= \infty$. 
We see from \eqref{eq-ode11} and Lemmas 
\ref{lem-blowup1} and 
\ref{lem-blowup2}
that 
$y_{R}^{\prime \prime}(t) < 0$ 
for all $t < (-\infty, 0)$. 
We claim that 
\begin{equation}\label{eq-ode12}
y_{R}^{\prime}(t) < 0 
\qquad \mbox{for all $t < 0$}. 
\end{equation}
Suppose to the contrary that 
\eqref{eq-ode12} fails. 
Then, one of the following two cases must occur:
there exist 
$t_{0} < 0$ and $\varepsilon_{0} > 0$ such that 
$y_{R}^{\prime}(t_{0}) > \varepsilon_{0}$, 
or
there exists $\widetilde{t}_{0} < 0$ 
such that $y_{R}^{\prime}(\widetilde{t}_{0}) = 0$. 
If the latter case occurs, 
we see from 
$y_{R}^{\prime \prime}(t) < 0$ 
for all $t < (-\infty, 0)$ that 
there exists $\widetilde{t}_{1} \in (\widetilde{t}_{0}, 0)$ 
and $\widetilde{\varepsilon}_{1} > 0$ such that 
$y_{R}^{\prime}(\widetilde{t}_{1}) > \widetilde{\varepsilon}_{1}$. 
Thus, it suffices to consider only the former case. 

We see from $y_{R}^{\prime \prime}(t) < 0$ 
for all $t < (-\infty, 0)$ and $y_{R}^{\prime}(t_{0}) > \varepsilon_{0}$ that 
$y_{R}^{\prime}(t) > \varepsilon_{0}$ 
for all $t < t_{0}$. 
Then, it follows from the fundamental theorem 
of calculus that 
\[
y_{R}(t_{0}) - y_{R}(t) 
= \int_{t}^{t_{0}} y_{R}^{\prime}(s) ds 
> \varepsilon_{0}(t_{0} - t). 
\]
This yields that 
\[
y_{R}(t) < y_{R}(t_{0}) - \varepsilon_{0} (t_{0} - t) 
= y_{R} (t_{0}) - \varepsilon_{0} t_{0} 
+ \varepsilon_{0} t 
\to - \infty \qquad (t \to - \infty), 
\]
which contradicts the positivity of $y_{R}(t)$. 
Thus, \eqref{eq-ode12} holds. 

Then, we see from \eqref{eq-ode12} that $y_{R}(t)$ 
is a decreasing function. 
In addition, we know from Lemma \ref{lem-blowup1} that 
$\lim_{t \to -\infty} \text{dist}_{H^{1}} 
(\psi(t), \mathcal{O}(Q_{\omega})) = 0$. 
These imply that 
\[
y_{R}(0) \leq \lim_{t \to - \infty} 
y_{R}(t) = \int_{\R^{3}} R^{2} 
\varphi(\frac{x}{R}) |Q_{\omega}|^{2} dx. 
\] 
Letting $R$ go to infinity, 
we obtain 
\[
\int_{\R^{3}} |x|^{2} |\psi_{0}|^{2} dx 
\leq \int_{\R^{3}} |x|^{2} |Q_{\omega}|^{2} dx 
< \infty. 
\] 
Then, from Lemma \ref{U-blowup}, we see that 
$T_{\max}^{-} < \infty$, which is 
absurd. 
Therefore, we conclude the desired result.  
\end{proof}

\section{Analysis on $\mathcal{BA}_{\omega, +}$}
In this section, we investigate the asymptotic behavior 
of solutions which start from $\mathcal{BA}_{\omega, +}$. 
\subsection{Convergence to the orbit of the ground state}
We can prove that 
if a solution with  
\eqref{nls} $\psi|_{t = 0} 
= \psi_{0} \in 
\mathcal{BA}_{\omega, +}$
does not scatter in positive time direction, 
the solution converges to 
the orbit of the ground state 
exponentially. 
More precisely, we obtain the following:
\begin{proposition}
\label{prop-conv-p}
Let $\omega_{3} > 0$ be the 
constant given in Lemma 
\ref{tech-lem1} and 
assume that $\omega \in 
(0, \omega_{3})$. 
Suppose a solution $\psi$ to \eqref{nls} 
with $\psi|_{t = 0} = \psi_{0} \in 
\mathcal{BA}_{\omega, +}$ 
does not scatter for positive time. 
Then, there exist constants 
$C > 0$ and $c> 0$
such that 
\begin{equation} \label{expo-decay2}
\text{dist}_{H^{1}} 
(\psi(t), \mathcal{O}(Q_{\omega})) 
\leq Ce^{- ct} 
\qquad \mbox{for all $t > 0$}. 
\end{equation}
\end{proposition}
The proof of Proposition 
\ref{prop-conv-p} is similar 
to that of Proposition 
\ref{prop-conv-n}. 
Thus, we omit the proof. 
\subsection{Scattering in negative time direction}
Secondly, we shall show that if a solution starts from 
$\mathcal{BA}_{\omega, +}$ does not scatter in 
positive time direction, the solution must scatter 
in negative one: 
\begin{proposition}
\label{prop-scatter}
Assume that $\omega \in 
(0, \omega_{3})$, 
where $\omega_{3} > 0$ is 
the constant given in Lemma 
\ref{tech-lem1}. 
Suppose that a solution $\psi$ to \eqref{nls} with $\psi|_{t = 0} 
= \psi_{0} \in 
\mathcal{BA}_{\omega, +}$ does not scatter 
in positive time direction. 
Then, $\psi$ exists on $(-\infty, \infty)$ and 
scatters for negative one, that is, 
there exists $\varphi_{-} \in H^{1}
(\R^{3})$ such that 
  \[
  \lim_{t \to -\infty} 
  \|\psi(t) - e^{i t \Delta} 
  \varphi_{-}\|_{H^{1}} = 0. 
  \]
\end{proposition}
To prove Proposition \ref{prop-scatter}, 
we need the following lemma: 
\begin{lemma}
\label{lem-scatter1}
Assume that $\omega \in 
(0, \omega_{3})$, 
where $\omega_{3} > 0$ is 
the constant given in Lemma 
\ref{tech-lem1}. 
Suppose that a solution $\psi$ to \eqref{nls} with $\psi|_{t = 0} 
= \psi_{0} \in 
\mathcal{BA}_{\omega, +}$ does not scatter 
in negative time direction. 
Then, 
$\lim_{t \to -\infty} \text{dist}_{H^{1}} 
(\psi(t), \mathcal{O}(Q_{\omega})) = 0$ and 
there exist 
a constant $C_{2} > 0$such that 
\begin{equation} \label{eq-ode14}
\mathcal{K}(\psi(t)) 
\geq C_{2} \|\eta(t)\|_{H^{1}} 
\qquad \mbox{for all $t \in 
(-\infty, 0)$}. 
\end{equation}
\end{lemma}
We can prove Lemma 
\ref{lem-scatter1} by a similar 
argument in the proof of 
Lemma \ref{lem-blowup1}. 
Thus, we omit the proof.  
We are now in the position to 
prove Proposition \ref{prop-scatter}. 
\begin{proof}[Proof of Proposition 
\ref{prop-scatter}]
We will follow the argument of \cite[Section 6.4]{DR2010}. 
Suppose that a solution $\psi$ does not scatter 
in both positive and negative time direction.  
Let $y_{R}$ be the function 
given by \eqref{eq-ode15}. 
From the fundamental theorem 
of calculus, we obtain 
\begin{equation} \label{eq-ode16}
\int_{\sigma}^{\tau} 
y_{R}^{\prime \prime}(t) dt 
= y_{R}^{\prime}(\tau) 
- y_{R}^{\prime}(\sigma).
\end{equation}
\eqref{eq-ode11} together with 
\eqref{eq-ode10} and 
\eqref{eq-ode14} yields that 
\begin{equation}\label{eq-ode17}
\int_{\sigma}^{\tau} 
y_{R}^{\prime \prime}(t) dt 
\geq \frac{C_{2}}{2} 
\int_{\sigma}^{\tau} 
\|\eta(t)\|_{H^{1}} dt >0 
\qquad \mbox{for all 
$\sigma < 0 < \tau$}. 
\end{equation}
Note that 
\begin{equation*}
\begin{split}
y_{R}^{\prime}(t) 
& = 2 R \text{Im}\; 
\int_{\R^{3}} \overline{\eta}(t, x) 
\nabla \varphi(\frac{x}{R}) 
\cdot \nabla Q_{\omega}(x) dx 
+ 2 R 
\text{Im}\; \int_{\R^{3}} 
Q_{\omega}(x) 
\nabla \varphi(\frac{x}{R}) 
\cdot \nabla \eta(t, x) dx \\
& \quad + 2 R \text{Im}\; 
\int_{\R^{3}} 
\overline{\eta(t, x)} 
\nabla \varphi(\frac{x}{R}) 
\cdot \nabla \eta(x) dx. 
\end{split}
\end{equation*}
From this and \eqref{expo-decay2}, we obtain 
\begin{equation} \label{eq-ode18}
|y_{R}^{\prime}(t)| 
\leq CR (\|\eta(t)\|_{H^{1}} 
+ \|\eta(t)\|_{H^{1}}^{2}) 
\to 0
\qquad \mbox{as $t \to \pm \infty$}. 
\end{equation}
Letting $\sigma$ go to $-\infty$ 
and $\tau$ to $+\infty$ 
in \eqref{eq-ode16}, we have, 
by \eqref{eq-ode17},  
\eqref{eq-ode18} and Lemmas \ref{lem-scatter1} 
and Lemma \ref{lem-blowup2}, that  
\[
\frac{C_{2}}{2}
\int_{-\infty}^{\infty} 
\|\eta(t)\|_{H^{1}}dt 
\leq \lim_{\sigma \to -\infty} 
|y^{\prime}(\sigma)| 
+ \lim_{\tau \to \infty} 
|y^{\prime}(\tau)| 
= 0. 
\]
This implies that $\eta(t) = 0$ 
for all $t \in \R$. 
However, this contradicts 
the assumption 
$\mathcal{K}(\psi(0)) = 
\mathcal{K}(\psi_{0}) > 0$. 
This completes the proof. 
\end{proof}
\section{Construction of special 
solutions}
In this section, we introduce 
Strichartz-type spaces and give the 
proof of Theorem 
\ref{main-sp}. 
First, we recall the Strichartz estimate:
\begin{definition}
We say that a pair of $(q, r)$ is \textit{$L^{2}$-admissible} 
if 
\begin{equation*}
\frac{1}{r} = \frac{3}{2} \left(\frac{1}{2} - \frac{1}{q}\right). 
\end{equation*}
\end{definition}

\begin{lemma}[Strichartz estimate] \label{Strichartz-est}
{\rm (i)}
For any $L^{2}$-admissible pair $(q, r)$, we have 
\begin{equation}\label{Stri-est-1}
\|e^{i t \Delta} u\|_{L^{r}(\mathbb{R}, L^{q})} \lesssim 
\|u\|_{L^{2}}. 
\end{equation} 
{\rm (ii)}
For any admissible pairs $(q_{1}, r_{1})$ and $(q_{2}, r_{2})$, we have 
\begin{equation} \label{Stri-est-2}
\left\|\int_{t_{0}}^{t} e^{i (t-t^{\prime}) \Delta} f(t^{\prime}) dt^{\prime} 
\right\|_{L^{r_{1}}(\mathbb{R}, L^{q_{1}})} \lesssim 
\|f\|_{L^{r_{2}^{\prime}}(\mathbb{R}, L^{q_{2}^{\prime}})} \qquad 
\mbox{for any $t_{0} \in \mathbb{R}$}, 
\end{equation}
where $q'$ and $r'$ denote the H\"{o}lder conjugates of $q$ and $r$ respectively.
\end{lemma}

We shall 
use the following Strichartz-type spaces: 
\begin{align*}
&St(I):=L_{t}^{\infty}L_{x}^{2}
(I \times \mathbb{R}^{3}) 
\cap L_{t}^{2}L_{x}^{6}
(I \times \mathbb{R}^{3}), 
\qquad 
V(I)
:= L_{t}^{10}L_{x}^{\frac{30}
{13}}(I \times \mathbb{R}^{3}), 
\\[6pt]
& 
V_{3}(I)
:= L_{t}^{5}L_{x}^{\frac{30}
{11}}(I \times \mathbb{R}^{3}), 
\qquad 
W (I)
:= L^{10}_{t} L_{x}^{10}
(I \times \mathbb{R}^{3}). 
\end{align*}

We define the norm of 
$St(I)$ for an interval $I$ by 
\[
\|u\|_{St(I)} := 
\sup_{\mbox{$(q, r)$: 
$L^{2}$-admissible}} 
\|u\|_{L_{t}^{r}L_{x}^{q}(I \times 
\R^{3})}. 
\]
From the definition, we see that 
\begin{equation}\label{Stri-est-3}
\|u\|_{V(I)}, 
\|u\|_{V_{3}(I)} \leq 
\|u\|_{St(I)}.
\end{equation}
We also use the following function 
space: 
\begin{equation*}
N(I):= L_{t}^{\frac{5}{3}}
L_{x}^{\frac{30}{23}}
(I \times \R^{3}), 
\end{equation*}
which is the dual space of 
the Strichartz space 
$L_{t}^{\frac{5}{2}}
L_{x}^{\frac{30}{7}}
(I \times \R^{3})$. 

By a similar argument to 
\cite[Lemma 6,1]{DM2009} 
(see also \cite[Proposition 3.4]
{DR2010}), 
we can construct a family of 
approximate solutions to \eqref{nls}. 
More precisely, we shall show the 
following:
\begin{proposition}\label{prop-approx}
Let $\omega_{2} > 0$ be the constant 
given by Theorem \ref{thm2-1}. 
For any 
$\omega \in (0, \omega_{2})$ and 
$A \in \R$, there exists a sequence 
$\{\mathcal{Z}_{j, \omega}^{A}
\}_{j \in \mathbb{N}}$ 
of functions in $\mathcal{S}(\R^{3}) 
\setminus \{0\}$ such that 
$\mathcal{Z}_{1, \omega}^{A} = A 
\mathcal{Y}_{\omega, -}$ and if 
$k \geq 1$ and $\mathcal{V}_{k, \omega}^{A} 
:= \sum_{j = 1}^{k} 
e^{- j e_{\omega} t} 
\mathcal{Z}_{k, \omega}^{A}$, 
then we have 
\begin{equation} \label{eq3-1}
\p_{t} 
\mathcal{V}_{k, \omega}^{A} + i \mathcal{L}_{\omega} 
\mathcal{V}_{k, \omega}^{A} 
= R (\mathcal{V}_{k, \omega}^{A}) 
+ O(e^{- (k + 1) e_{\omega} t}) 
\qquad \mbox{in $\mathcal{S}(\R^{3})$ 
as $t \to \infty$}. 
\end{equation}
\end{proposition}
By using Proposition 
\ref{prop-approx} and 
the contraction argument, 
we can construct a solution 
$U_{\omega}^{A}$ to 
\eqref{nls} which is close to  
$e^{i \omega t} (Q_{\omega} 
+ \mathcal{V}_{k, \omega}^{A})$. 
\begin{proposition}\label{prop-spec}
Let $\omega_{2} > 0$ be the constant 
given by Theorem \ref{thm2-1}. 
For any 
$\omega \in (0, \omega_{2})$ and 
$A \in \R$,
there exists $k_{0} > 0$ such that 
for any $k \geq k_{0}$, 
there exist sufficiently large 
$t_{0} = t_{0}(\omega, A) > 0$ and 
the radial solution $U_{\omega}^{A} 
\in C^{\infty} ([t_{0}, \infty), H^{\infty} 
(\R^{3}))$ to \eqref{nls} satisfying 
the following: 
\begin{equation}
\label{const-eq1}
\|\langle \nabla \rangle
(U_{\omega}^{A} - e^{i \omega t} 
(Q_{\omega} 
+ \mathcal{V}_{k, \omega}^{A}))
\|_{St(t, \infty)}
\leq e^{- (k + \frac{1}{2}) 
e_{\omega} t}
\qquad (t \geq t_{k}). 
\end{equation}
Furthermore, $U_{\omega}^{A}$ 
is a unique solution to \eqref{nls} 
satisfying \eqref{const-eq1} for large 
$t > 0$. 
Finally, 
$U_{\omega}^{A}$ is independent of 
$k \in \mathbb{N}$ and satisfies 
\begin{equation} 
\label{const-eq1-1}
\|U_{\omega}^{A} - e^{i \omega t}(Q_{\omega} + A e^{- e_{\omega}t} 
\mathcal{Y}_{\omega, -})\|_{H^{1}} 
\leq C e^{- \frac{3}{2}e_{\omega} t} 
\qquad 
\mbox{for all $t \geq t_{0}$}. 
\end{equation}
\end{proposition}
We can prove Proposition 
\ref{prop-spec} in a way 
similar to 
\cite[Proposition 6.3]{DM2009}. 
Thus, we omit the proof. 
We are now in a position to prove 
Theorem \ref{main-sp}. 
\begin{proof}[Proof of Theorem 
\ref{main-sp}]
We put 
\begin{equation}\label{def-spe}
Q_{\omega}^{\pm} (t, x)
:= e^{- i \omega t_0} U_{\omega}^{\pm1} 
(t + t_0, x). 
\end{equation}
Then, we see from $\mathcal{Y}_{\omega, -} = 
\overline{\mathcal{Y}}_{\omega, +}$, 
\eqref{eq-egf1}, \eqref{const-eq1-1} and 
\eqref{def-spe} that 
	\[
	\begin{split}
	Q_{\omega}^{\pm} (t, x) 
	& = e^{- i \omega t_{0}} U_{\omega}^{\pm 1} 
	(t + t_{0}, x) \\
	& = e^{- i \omega t_{0}}  e^{i \omega (t + t_{0})}
	(Q_{\omega} \pm e^{- e_{\omega} 
	(t + t_0)} \mathcal{Y}_{\omega, -} 
	+ \Gamma^{\pm}(t, x)) \\
	& = e^{i \omega t} 
	(Q_{\omega} \pm e^{- e_{\omega} 
	(t + t_0)} 
	\mathcal{Y}_{\omega, 1} 
	 \mp i e^{- e_{\omega} 
	(t + t_0)} 
	\mathcal{Y}_{\omega, 2} 
	+ \Gamma^{\pm}(t, x)),   
	\end{split}
	\]
where $\Gamma^{\pm} \in C^{\infty} 
([t_0, \infty), H^{\infty} (\R^{3}))$ with 
$\|\Gamma^{\pm}(t)\|_{H^{1}} = O(e^{- 
\frac{3}{2} e_{\omega} t})$. 	
Then, by a similar argument in \eqref{eq-egf10-1}~
\footnote{$\lambda_{1}(t), \lambda_{2}(t)$ 
and $\Gamma(t)$ in \eqref{eq-egf10-1} correspond 
to $\pm \frac{e^{- e_{\omega}(t + t_{0})}}{2}, 
\mp \frac{e^{- e_{\omega}(t + t_{0})}}{2}$ and 
$\Gamma^{\pm}(t)$, respectively}, 
we obtain 
	\[
	\begin{split}
	\mathcal{K}(Q_{\omega}^{\pm}) 
	& 
	= \mp e_{\omega} \frac{e^{- e_{\omega} 
	(t + t_{0})}}{2} (Q_{\omega}, \mathcal{Y}_{\omega, 2})_{L^{2}_{\text{real}}}
	\mp 2 e^{- e_{\omega} 
	(t + t_{0})} (Q_{\omega}^5, \mathcal{Y}_{\omega, 1}
	)_{L^{2}_{\text{real}}} 
	- (Q_{\omega}^{3}, \Gamma^{\pm}(t))_{L^{2}_{\text{real}}} \\
	& \quad 
	- 4 (Q_{\omega}^{5}, \Gamma^{\pm}(t))_{L^{2}_{\text{real}}} 
	+ O(e^{-2 e_{\omega} t}). 
	\end{split}
	\]
By \eqref{eq-egf10}, \eqref{eq-ejec4} and 
$\|\Gamma^{\pm}(t)\|_{H^{1}} = O(e^{- 
\frac{3}{2} e_{\omega} t})$, 
we have 
	\[
	\begin{split}
	\pm \mathcal{K}(Q_{\omega}^{\pm}) 
	\geq e_{\omega} 
	\frac{e^{- e_{\omega} 
	(t + t_{0})}}{8} 
	|(Q_{\omega}, \mathcal{Y}_{\omega, 2})_{L^{2}_{\text{real}}}| 
	> 0
	\end{split} 
	\]
for sufficiently large $t > 0$.  
Thus, we see that 
$Q_{\omega}^{\pm} \in \mathcal{BA}_{\omega, \pm}$ 
which satisfy 
  \[
  \text{dist}_{H^{1}}(Q_{\omega}^{\pm}(t), \mathcal{O} 
(Q_{\omega})) \leq C e^{- e_{\omega} t} 
\qquad \mbox{for all $t \geq 0$}.
  \]
Then, it follows from Propositions \ref{K-blowup} and 
\ref{prop-scatter} that 
$Q_{\omega}^{+}$ blows up in finite negative 
time and $Q_{\omega}^{-}$ is globally defined and 
scatters for negative time. 
This completes the proof. 
\end{proof}
\section{Uniqueness and proof of 
Theorem \ref{main-class}}
In this section, we shall show that 
a solution which converges to the 
orbit of the ground state must be 
the special one obtained in Proposition 
\ref{prop-spec}.  
After that, we will give the proof of 
Theorem \ref{main-class}. 
Our first aim in this 
section is to prove the following: 
\begin{proposition}\label{prop-unique}
Assume that $\omega \in 
(0, \omega_{3})$, 
where $\omega_{3} > 0$ is the 
constant given in Lemma 
\ref{tech-lem1}. 
Suppose that a solution $\psi$ to \eqref{nls} with $\psi|_{t = 0} 
= \psi_{0} \in 
\mathcal{BA}_{\omega}$ 
satisfies 
\begin{equation} \label{eq-uni1}
\text{dist}_{H^{1}} 
(\psi(t), \mathcal{O}(Q_{\omega}))  
\leq C e^{-c t} \qquad 
(t \geq t_{0}) 
\end{equation}
for some $c, C \in \R$ and $t_{0} 
 >0$.  
Then, there exists $A_0 \in \R$ and 
$\theta_{0} \in \R$
such that $\psi 
= e^{i \theta_{0}} U_{\omega}^{A_0}$, 
where $U_{\omega}^{A_0}$ is the solution to 
\eqref{nls} defined in Proposition 
\ref{prop-spec}. 
\end{proposition}
\subsection
{Exponentially small solution 
to the linearized equation}
In this subsection, 
in order to prove Proposition \ref{prop-unique}, 
we consider 
$\eta \in C^{0}([t_{0}, \infty), 
H^{1}(\R^{3}))$ and 
$g \in C^{0}([t_{0}, \infty), 
L^{2}(\R^{3}))$
such that 
\begin{equation}
\p_{t} \eta + 
i \mathcal{L}_{\omega} \eta 
= g \qquad 
\mbox{in $(t, x) \in 
(t_{0}, \infty) \times \R^{3}$}, 
\label{eq-uni2} 
\end{equation}
\begin{equation}
\|\eta(t)\|_{H^{1}} \leq C e^{- 
\gamma_{1} t} \qquad 
\qquad 
(t \geq t_{0}),  
\label{eq-uni3}
\end{equation}
\begin{equation}
\|\langle \nabla \rangle 
g\|_{N(t, \infty)} 
+ \|g(t)\|_{L^{\frac{6}{5}}(\R^{3})}
\leq C 
e^{-\gamma_{2}t} 
\qquad (t \geq t_{0}), 
\label{eq-uni313}
\end{equation}
where 
\[
0 < \gamma_{1} < \gamma_{2}. 
\]
We shall show the following:
\begin{proposition}
\label{pro-uni1}
Assume that 
$\eta, g$ satisfy 
\eqref{eq-uni2}--\eqref{eq-uni313}. 
Then, the followings hold: 
\begin{enumerate}
\item[\textrm{(i)}] 
if 
$\gamma_2 \leq 
e_{\omega}$, 
$\|\eta(t)\|_{H^{1}} 
+ 
\|\langle \nabla \rangle 
\eta\|_{St(t, \infty)}
\leq 
C e^{- \gamma_{2}^{-} t}$, 
\item[\textrm{(ii)}] 
if $
\gamma_{2} > e_{\omega}$, 
there exists $A \in \R$ such that 
$\eta(t) = A e^{- e_{\omega}t} 
\mathcal{Y}_{\omega, -} + w(t)$ 
with $\|w(t)\|_{H^{1}} 
+ 
\|\langle \nabla \rangle w
\|_{St(t, \infty)}
\leq C 
e^{- \gamma_{2}^{-} t}$. 
\end{enumerate}
\end{proposition}
The proof of Proposition \ref{pro-uni1} 
is similar to that of Proposition 5.9 
in \cite{DM2009}. 
However, since we still employ the symplectic 
decomposition, the detail is 
a bit different. 
Therefore, we give the proof for the sake 
of completeness. 

To prove Proposition \ref{pro-uni1}, 
we need several preparations. 
Note that 
\eqref{eq-uni2} is equivalent to 
\begin{equation}\label{eq-mod11}
i \p_{t} \eta + \Delta \eta - \omega \eta 
+ \mathcal{V} (\eta) = ig, 
\end{equation} 
where 
\begin{equation}\label{linear-eq}
\begin{split}
\mathcal{V}(\eta) 
& = Q_{\omega}^{2} 
(2\eta + \overline{\eta}) 
+ Q_{\omega}^{4} 
(3\eta + 2 \overline{\eta}) \\
& = 
3 Q_{\omega}^{2} 
(\text{Re} \; \eta) 
+ i Q_{\omega}^{2} 
(\text{Im}\; \eta) 
+ 5 Q_{\omega}^{4} 
(\text{Re} \; \eta) 
+ i Q_{\omega}^{4} 
(\text{Im}\; \eta). 
\end{split}
\end{equation} 
 As in \cite[Lemma 5.5]{DM2009}, 
 we obtain the following: 
\begin{lemma}[Linear estimate]
\label{pre-est1}
\begin{enumerate}
\item[\textrm{(i)}]
Let $f \in L^{6}(\R^{3})$. 
Then, there exists a constant 
$C_{1} > 0$ such that
\[
\|\mathcal{V}(f)
\|_{L^{\frac{6}{5}}} 
\leq C_{1} \|f\|_{L^{6}}. 
\]
\item[\textrm{(ii)}] 
Let $I$ be a finite time interval 
of length $I$ and $f \in W(I)$ 
such that $\nabla f \in V(I)$. 
Then, there exists $C_{2} > 0$ 
independent of $I, f$ and $g$ 
such that 
\begin{align}
& \|f\|_{W(I)} \leq C_{2} 
\|\nabla f\|_{V(I)},
\label{linear-est00} \\
& \|\langle \nabla \rangle 
\mathcal{V}(f)
\|_{N(I)} \leq |I|^{\frac{2}{5}} 
\|\langle \nabla \rangle f\|_{V_{3}(I)}. 
\label{linear-est0}
\end{align}
\end{enumerate}
\end{lemma}
\begin{proof}
We can obtain \textrm{(i)} 
by the H\"{o}lder inequality. 
\eqref{linear-est00} follows from 
the Sobolev inequality. 
We can also obtain \eqref{linear-est0} 
by the H\"{o}lder inequality. 
\end{proof}
\begin{lemma} \label{uni-lem1}
For any finite time-interval $I$, 
of length $|I|$, and any functions 
$g$ and $\eta$ such that 
$g \in L^{\infty} (I, L^{\frac{6}{5}}), 
\langle \nabla \rangle g \in N(I), 
g \in L^{\infty} (I, L^{2}(\R^{3})), 
\eta \in L^{\infty} 
(I, L^{6}(\R^{3}))$ 
and $\langle \nabla \rangle \eta \in 
L^{\frac{5}{2}}(I, L^{
\frac{30}{7}}(\R^{3}))$, 
we have 
\begin{equation} \label{eq-uni29}
\begin{split}
\int_{I} 
\biggl| \langle \mathcal{L}_{\omega}
g(t), \eta(t)\rangle \biggl| dt 
& \leq C \biggl[
\|\langle \nabla \rangle g\|_{N(I)} 
\|\langle \nabla \rangle \eta
\|_{L^{\frac{5}{2}}
(I, L^{\frac{30}{7}}(\R^{3}))} \\
& \qquad \quad + |I| \|g\|_{L^{\infty} (I, L^{\frac{6}{5}}(\R^{3}))} 
\|\eta\|_{L^{\infty}(I, L^{6})}
\biggl]. 
\end{split}
\end{equation} 
\end{lemma}
\begin{proof}
We have 
\[
\langle \mathcal{L}_{\omega}
g(t), \eta(t)\rangle
= a(t) + b(t), 
\]
where 
\[
a(t) = \text{Re}\; \int_{\R^{3}} \nabla g(t) 
\cdot \nabla \overline{\eta(t)} dx 
+ \omega \text{Re}\; \int_{\R^{3}} 
g(t) \overline{\eta(t)} dx, 
\]
\[
b(t) = - 
\int_{\R^{3}} 
(3Q_{\omega}^{2} + 5 Q_{\omega}^{4}) 
g_{1}(t) \eta_{1}(t) dx - 
\int_{\R^{3}} 
(Q_{\omega}^{2} + Q_{\omega}^{4}) 
g_{2}(t) \eta_{2}(t) dx. 
\]
Here, $g_{1} = \text{Re}\; g, 
\eta_{1} = \text{Re}\; \eta$ and 
$g_{2} = \text{Im}\; g, 
\eta_{2} = \text{Im}\; \eta$.
By the H\"{o}lder inequality, we obtain 
\[
\int_{I} |a(t)| dt 
\leq C \|\langle \nabla \rangle 
g\|_{N(I)} 
\|\langle \nabla \eta \rangle
\|_{L_{t}^{\frac{5}{2}} 
(I, L_{x}^{\frac{30}{7}})}, 
\]
\[
|b(t)| \leq 
C \|g(t)\|_{L_{x}^{\frac{6}{5}}(\R^{3})}
\|\eta(t)\|_{L_{x}^{6}(\R^{3})} 
(\|Q_{\omega}\|_{L^{\infty}}^{2} 
+ \|Q_{\omega}\|_{L^{\infty}}^{4}). 
\]
Integrating the estimate on $b(t)$ over $I$, 
we get the conclusion. 
\end{proof} 
We recall the following lemma which 
is obtained by Duyckaerts and 
Merle~\cite[Claim 5.8]{DM2009}.
\begin{lemma}[Sums of exponential ]\label{sp-thm1}
Let $t_{0} > 0, p \in [1, \infty), 
a_{0} < 0$, $E$ a normed vector 
space, and 
$f \in L_{\text{loc}}^{p}
(t_{0}, \infty; E)$ such that 
there exist $\tau_{0} > 0, 
C_{0} > 0$ satisfying 
\begin{equation*} 
\|f\|_{L^{p} (t, t + \tau_{0}; E)} 
\leq C_{0} e^{a_{0} t} 
\qquad (t \geq t_{0}). 
\end{equation*}
Then, we have 
\begin{equation*}
\|f\|_{L^{p} (t, \infty; E)} 
\leq \frac{C_{0} e^{a_{0} t}}
{1 - e^{a_{0} \tau_{0}}} 
\qquad 
(t \geq t_{0}). 
\end{equation*}
\end{lemma} 
Using Lemma \ref{sp-thm1}, 
we shall show the following:
\begin{lemma}
\label{uni-lem-4}
Let $\eta$ be a solution to 
\eqref{eq-uni2}. 
Assume that $\eta$ satisfies 
\begin{equation} \label{non-est-11}
\|\eta(t)\|_{H^{1}} \leq C_{1} 
e^{- \gamma t} 
\qquad 
(t > 0)
\end{equation}
for some $C_{1} > 0$ and 
$\gamma \in (0, \gamma_{2})$. 
Then, we have 
\[
\|\langle \nabla \rangle \eta
\|_{St (t, \infty)} + 
\|\eta\|_{W(t, \infty)} \leq C_{2} 
e^{- \gamma t}
\qquad 
(t > 0)
\]
for some $C_{2} > 0$. 
\end{lemma}

\begin{proof}
We shall show Lemma \ref{uni-lem-4} following 
\cite[Lemma 5.7]{DM2009}. 
From \eqref{eq-mod11}, one has 
\begin{equation} \label{non-est-101}
i \p_{t} \langle \nabla \rangle \eta 
+ \Delta \langle \nabla \rangle \eta 
- \omega \langle \nabla \rangle \eta
+ \langle \nabla \rangle 
(\mathcal{V} (\eta) 
- i g) = 0.  
\end{equation}
Let $t$ and $\tau$ such that 
$t > 0$ and $0 < \tau < 1$. 
By the Strichartz estimates 
\eqref{Stri-est-1}, \eqref{Stri-est-2}, 
\eqref{non-est-11}, 
\eqref{linear-est0}
and \eqref{eq-uni313}, 
one has 
\begin{equation}\label{non-est-12}
\begin{split}
\|\langle \nabla \rangle \eta\|_{V_{3}
(t, t + \tau)} 
& \leq 
C (\|\eta(t)\|_{H^{1}} 
+ \|\langle \nabla \rangle 
(\mathcal{V} (\eta) \|_{N(t, t + \tau)}
+ \|\langle \nabla \rangle g\|_{N(t, t + \tau)})
\\
& \leq C \bigg(e^{- \gamma t} 
+ \tau^{\frac{2}{5}} \|\langle \nabla 
\rangle \eta\|_{V_{3}(t, t+\tau)} 
+ e^{-\gamma_{2} t}\bigg). 
\end{split}
\end{equation}
Thus, we can take 
$\tau_{0} > 0$ sufficiently small so that 
\begin{equation}
\label{non-est-13} 
\|\langle \nabla 
\rangle \eta\|_{V_{3}
(t, t + \tau_{0})} 
\leq C e^{- \gamma t}. 
\end{equation} 
for $t > 0$.
Note that by the H\"{o}lder inequality, 
the assumption 
\eqref{non-est-11} and 
\eqref{non-est-13}, we obtain 
\begin{equation} \label{sp-eq3} 
\|\langle \nabla 
\rangle \eta\|_{V(t, t + \tau_{0})} 
\leq 
\|\eta\|_{L^{\infty}(
(t, t + \tau_{0}), H^{1})}
^{\frac{1}{2}}
\|\langle \nabla 
\rangle \eta
\|_{V_{3}(t, t + \tau_{0})}
^{\frac{1}{2}}
\leq C_{1} e^{- \gamma t}. 
\end{equation}
By Lemma \ref{sp-thm1} and 
the inequalities \eqref{non-est-13}
and \eqref{sp-eq3} , 
we have 
\begin{equation}
\label{non-est-14}
\|\langle \nabla 
\rangle \eta\|_{V (t, \infty)} + 
\|\langle \nabla 
\rangle \eta\|_{V_{3} (t, \infty)} 
\leq C e^{- \gamma t} 
\qquad (t > 0). 
\end{equation}
Then, by 
the Sobolev inequality 
\eqref{linear-est00}, we see that 
$\|\eta\|_{W(t, \infty)} \leq C_{2} 
e^{- \gamma t}$ for some $C_{2} > 0$. 

Next, we shall estimate 
$\|\langle \nabla \rangle \eta
\|_{St(t, \infty)}$. 
By the Strichartz estimates 
\eqref{Stri-est-1}, 
\eqref{Stri-est-2}, 
and the estimate similar to 
\eqref{non-est-12}, 
we obtain 
\[
\begin{split}
\|\langle \nabla \rangle 
\eta\|_{St(t, t+ \tau)} 
& 
\leq 
C (e^{- \gamma t} 
+ \tau^{\frac{2}{5}} 
\|\langle \nabla 
\rangle \eta\|_{V_{3}(t, t + \tau)} 
+ 
\|\langle \nabla \rangle g\|_{N(t, t + \tau)}) \\
& 
\leq 
C (e^{- \gamma t} 
+ \tau^{\frac{2}{5}} 
\|\langle \nabla 
\rangle \eta\|_{St(t, t+ \tau)} 
+ e^{- \gamma_{2} t}).
\end{split}
\] 
This implies that 
$\|\langle \nabla \rangle 
\eta\|_{St(t, t + \tau)} 
\leq C e^{- \gamma t}$ 
for sufficiently small $\tau > 0$. 
By Lemma \ref{sp-thm1}, 
we get 
$\|\langle \nabla \rangle 
\eta\|_{St(t, \infty)} \leq C_{2} 
e^{- \gamma t}$ 
for some $C_{2} > 0$. 
This completes the proof. 
\end{proof}

We are now in a position to prove 
Proposition \ref{pro-uni1}. 
\begin{proof}[Proof of Proposition 
\ref{pro-uni1}]

\par
\textbf{(Step 1). }
We claim the following: 
\begin{align}
& 
\frac{d \lambda_{+}}{d t}(t) 
- e_{\omega} \lambda_{+}(t) 
= \Omega(g, \mathcal{Y}_{\omega, 
-}), 
\label{eq-uni8}
\\
& 
\frac{d \lambda_{-}}{d t}(t) 
+ e_{\omega} \lambda_{-}(t) 
= - \Omega(g, 
\mathcal{Y}_{\omega, +}), 
\label{eq-uni8-1}
\\
& \frac{d}{d t} 
\langle \mathcal{L}_{\omega} 
\eta(t), \eta(t) \rangle_{H^{-1}, H^{1}} 
= 2 \langle \mathcal{L}_{\omega} 
\eta(t), g\rangle_{H^{-1}, H^{1}},  
\label{eq-uni9-1}
\end{align}
where $\lambda_{\pm}(t)$ are given by 
\eqref{eq-dec15}. 
It follows from 
\eqref{eq-dec15}, 
\eqref{eq-uni2} and 
$\mathcal{L}_{\omega} 
\mathcal{Y}_{\omega, -} 
= - i e_{\omega} \mathcal{Y}
_{\omega, -}$ that 
	\[
	\begin{split}
	\frac{d \lambda_{+}}{d t}(t) 
	= \Omega(\p_{t} \eta(t), 
	\mathcal{Y}_{\omega, -})
	= \Omega(- i \mathcal{L}
	_{\omega} \eta(t), \mathcal{Y}
	_{\omega, -}) + 
	\Omega(g, 
	\mathcal{Y}_{\omega, -}) 
	= e_{\omega} \lambda_{+}(t) 
	+ \Omega(g, \mathcal{Y}_{\omega, -}). 
	\end{split}
	\]
Similarly, by \eqref{eq-dec15}, 
\eqref{eq-uni2} and 
$\mathcal{L}_{\omega} 
\mathcal{Y}_{\omega, +} 
= i e_{\omega} \mathcal{Y}
_{\omega, +}$, we obtain 
	\[
	\begin{split}
	\frac{d \lambda_{-}}{d t}(t) 
	= - \Omega(\partial_{t} 
	\eta(t), \mathcal{Y}_{\omega, +}) 
	& 
    = - \Omega(- i 
	\mathcal{L}_{\omega} \eta(t), 
	\mathcal{Y}_{\omega, +}) 
	- \Omega(g, 
	\mathcal{Y}_{\omega, +}) \\
    & = - e_{\omega} \lambda_{-}(t) 
	- \Omega(g, 
	\mathcal{Y}_{\omega, +}).    
	\end{split}
	\]
Clearly, we have 
  \begin{equation}\label{eq-uni60}
  (\mathcal{L}_{\omega} u, - i 
  \mathcal{L}_{\omega}u)_{L^{2}_{\text{real}}} 
  = 0 
  \qquad \mbox{for all $u \in H^{2}(\R^{3})$}.
  \end{equation}
It follows form \eqref{eq-uni2} 
and \eqref{eq-uni60} that 
	\[
	\begin{split}
	\frac{d}{d t} \langle \mathcal{L}
	_{\omega} 
\eta(t), \eta(t) \rangle_{H^{-1}, H^{1}}   
= 2 \langle \mathcal{L}_{\omega} 
\eta(t), g \rangle_{H^{-1}, H^{1}}. 
\end{split}
	\]
This yields \eqref{eq-uni9-1}.

\textbf{(Step 2). 
}
We now claim the following: 
\begin{align} 
|\lambda_{+}(t)| 
\leq C e^{- \gamma_{2}t}, & \label{eq-uni31} \\
|\lambda_{-}(t)| 
\leq 
C e^{- \gamma_{2}^{-}t} 
& \qquad 
\mbox{if $\gamma_{2} 
\leq e_{\omega}$}, 
\label{eq-uni32}
\\
\mbox{there exists $A \in 
\R$ such that 
$|\lambda_{-}(t) - A 
e^{- e_{\omega} t}| 
\leq C e^{- \gamma_{2}t}$}
& \qquad 
\mbox{if $\gamma_{2} 
> e_{\omega}$}. 
\label{eq-uni321}
\end{align}
It follows from \eqref{eq-uni8} and \eqref{eq-uni313} 
that 
\[
\biggl|\frac{d}{d t} (e^{- e_{\omega} t} 
\lambda_{+}(t)) \biggl|
= e^{- e_{\omega} t} 
|\Omega(g, \mathcal{Y}_{\omega, -})| 
\leq C e^{-e_{\omega} t} \|g\|_{L^{\frac{6}{5}}(\R^{3})} 
\leq C e^{- (e_{\omega} + \gamma_{2}) t}. 
\] 
Integrating the above from $t$ to $\infty$, 
we obtain \eqref{eq-uni31}.

Next, we shall show \eqref{eq-uni32}. 
From \eqref{eq-uni8-1}, we obtain 
	\begin{equation} \label{eq-uni54}
	\biggl|
	\frac{d}{d t} 
	(e^{e_{\omega} t} \lambda_{-}
	(t))\biggl| 
    = e^{e_{\omega} t} 
|\Omega(g, \mathcal{Y}_{\omega, +})| 
\leq C e^{e_{\omega} t} \|g\|_{L^{\frac{6}{5}}(\R^{3})}
	\leq C e^{(e_{\omega} 
	- \gamma_{2})t}. 
	\end{equation}
Assume that $\gamma_{2} 
\leq e_{\omega}$. 
By \eqref{eq-uni54}, one has 
	\[
	|e^{e_{\omega} t} 
	\lambda_{-}(t)| 
	\leq 
	\begin{cases}
	e^{e_{\omega} t_0} 
	|\lambda_{-}(t_0)| 
	+ C e^{(e_{\omega} - 
	\gamma_{2}) t} 
	& 
	\qquad \mbox{if $\gamma_{2} 
    < e_{\omega}$}, \\
	e^{e_{\omega} t_0} 
	|\lambda_{+}(t_0)| 
	+ C(t - t_{0})
	& 
	\qquad \mbox{if $e_{\omega} 
	= \gamma_{2}$}. 
	\end{cases}
	\]	
This yields \eqref{eq-uni32}. 	
	
Next, we assume that 
$\gamma_{2} > e_{\omega}$. 
Then, we see that 
$\int_{t_0}^{\infty} 
e^{(e_{\omega} - \gamma_{2})
\tau} d \tau < \infty$. 	
It follows from \eqref{eq-uni54} that 
	\begin{equation} \label{eq-uni55}
	|e^{e_{\omega} s}
	\lambda_{-}(s) 
	- e^{e_{\omega} t}
	\lambda_{-}(t)| 
	\leq C \int_{t}^{s} 
	e^{(e_{\omega} 
	- \gamma_{2}) \tau} d\tau  
	\leq C \int_{t}^{\infty} 
	e^{(e_{\omega} 
	- \gamma_{2}) \tau} d\tau 
	\leq C e^{(e_{\omega} 
	- \gamma_{2})t}
	\end{equation}
for $s \geq t > t_{0}$. 	
Thus, there exists $A \in \R$ 
such that $\lim_{t \to \infty} 
e^{e_{\omega} t} \lambda_{-}(t) 
= A$. 	
In addition, letting 
$s \to \infty$ in \eqref{eq-uni55}, 
we have 
	\[
	|A - e^{e_{\omega} t}
	\lambda_{-}(t)| \leq 
	C e^{(e_{\omega} - 
	\gamma_{2})t}. 
	\]
This implies \eqref{eq-uni321}. 

\textbf{(Step 3)}. 
We next prove 
\begin{equation} \label{eq-uni34}
\|\Gamma(t)\|_{H^{1}} \leq 
C e^{- \frac{\gamma_{1} 
+ \gamma_{2}}{2} t},  
\end{equation}
where $\Gamma(t)$ is given by \eqref{eq-energy10}. 
We have, by 
\eqref{eq-dec12-1},  
\eqref{eq-dec14} and 
\eqref{eq-energy10}, that 
$\Omega(\Gamma(t), 
\mathcal{Y}_{\omega, +}) = 
\Omega(\Gamma(t), 
\mathcal{Y}_{\omega, -}) = 0$. 
Recall that 
$\Omega(\mathcal{Y}_{\omega, +}, 
\mathcal{Y}_{\omega, -}) = 1$ 
and 
$\Omega(\mathcal{Y}_{\omega, +}, 
\mathcal{Y}_{\omega, +}) 
= \Omega(\mathcal{Y}_{\omega, -}, 
\mathcal{Y}_{\omega, -})
= 0$. 
Using this, 
we get 
\begin{equation}\label{eq-uni19}
\begin{split}
\langle \mathcal{L}_{\omega} 
\eta(t), \eta(t) \rangle_{H^{-1}, H^{1}} 
& = 
e_{\omega} \lambda_{+}(t) 
\lambda_{-}(t) 
(i \mathcal{Y}_{\omega, +}, 
\mathcal{Y}_{\omega, -})
_{L_{\text{real}}^{2}} 
- e_{\omega} \lambda_{+}(t) 
\lambda_{-}(t) 
(i \mathcal{Y}_{\omega, -}, 
\mathcal{Y}_{\omega, +})
_{L_{\text{real}}^{2}} \\
& \quad 
+ \langle \mathcal{L}_{\omega} 
\Gamma(t), \Gamma(t) \rangle
_{H^{-1}, H^{1}} \\
& = -2 e_{\omega} 
\lambda_{+}(t) \lambda_{-}(t) 
+  
\langle \mathcal{L}_{\omega} 
\Gamma(t), \Gamma(t) \rangle
_{H^{-1}, H^{1}}.  
\end{split}
\end{equation}
By Lemmas \ref{uni-lem1} 
and \ref{uni-lem-4}, 
\eqref{eq-uni9-1}, \eqref{eq-uni313} 
and \eqref{eq-uni3}, we have 
	\[
	\begin{split}
	& \quad 
	\int_{t}^{t + 1} 
	\biggl|\frac{d}{d s} 
	\langle \mathcal{L}_{\omega} 
\eta(s), \eta(s) \rangle_{H^{-1}, H^{1}} \biggl| ds \\
& = 2 \int_{t}^{t + 1} 
\biggl|  
	\langle \mathcal{L}_{\omega} 
g(s), \eta(s) \rangle_{H^{-1}, H^{1}} \biggl| ds \\
& \leq 
C \left[
\|\langle \nabla \rangle g\|_{N
([t, t+1])} 
\|\langle \nabla \rangle \eta
\|_{L^{\frac{5}{2}}
([t, t+1], L^{\frac{30}{7}}(\R^{3}))} 
+ \|g\|_{L^{\infty}
([t, t+1], L^{\frac{6}{5}}(\R^{3}))} 
\|\eta\|_{L^{\infty}
([t, t+1], L^{6}(\R^{3}))}
\right] \\
& \leq C e^{- (\gamma_{1} + 
\gamma_{2})t}. 
	\end{split}
	\]
Then, from Lemma \ref{sp-thm1}, 
we obtain 
	\[
	\int_{t}^{\infty} 
	\biggl|\frac{d}{d s} 
	\langle \mathcal{L}_{\omega} 
\eta(s), \eta(s) \rangle_{H^{-1}, H^{1}} \biggl| ds 
\leq C e^{- (\gamma_{1} + 
\gamma_{2})t}. 
	\]
Since $\lim_{t \to \infty} 
\langle \mathcal{L}_{\omega} 
\eta(t), \eta(t) \rangle_{H^{-1}, H^{1}} = 0$, 
we see that 
	\begin{equation} \label{eq-uni40}
	\biggl|
	\langle \mathcal{L}_{\omega} 
\eta(t), \eta(t) \rangle_{H^{-1}, H^{1}} \biggl|
\leq \int_{t}^{\infty} 
\biggl| \frac{d}{d s} 
	\langle \mathcal{L}_{\omega} 
\eta(s), \eta(s) \rangle_{H^{-1}, H^{1}} \biggl| 
ds \leq C e^{- (\gamma_{1} + 
\gamma_{2})t}.
	\end{equation}
From Lemma \ref{lem-energy1}, \eqref{eq-uni31}, 
\eqref{eq-uni32},  
\eqref{eq-uni19} and \eqref{eq-uni40},  
we obtain 
	\begin{equation*} 
	\|\Gamma(t)\|_{H^{1}}^{2} 
	\leq 
	C \langle \mathcal{L}_{\omega} 
\Gamma(t), \Gamma(t) \rangle_{H^{-1}, H^{1}} 
\leq C e^{- (\gamma_{1} + 
\gamma_{2}) t}. 
	\end{equation*}
From this, 
we conclude that 
\eqref{eq-uni34} holds.

\textbf{(Step 4). 
} 
We conclude the proof. 
We first consider the case of 
$e_{\omega} \geq 
\gamma_{2}$ or the case of 
$\gamma_2 > e_{\omega}$ and 
$A = 0$. 
By the decomposition 
\eqref{eq-dec13} of $\eta$, 
\eqref{eq-uni31}, \eqref{eq-uni32}, 
\eqref{eq-uni321} and 
\eqref{eq-uni34},  
we have 
\begin{equation} 
\label{eq-uni35} 
\|\eta(t)\|_{H^{1}} 
\leq 
C e^{- \frac{\gamma_{1} 
+ \gamma_{2}}{2}t}. 
\end{equation} 
Iterating the above argument, 
we get the bound 
\[
\|\eta(t)\|_{H^{1}} \leq C_{1} 
e^{- 
\gamma_{2}^{-}}. 
\] 
From Lemma \ref{uni-lem-4}, 
we have 
$\|\langle \nabla \rangle 
\eta\|_{St(t, \infty)} \leq C_{2} 
e^{- \gamma_{2}^{-} t}$.

Secondly, we consider the case of 
$e_{\omega} < \gamma_{2}$. 
Then, by the decomposition 
\eqref{eq-energy1}, \eqref{eq-uni31}, 
\eqref{eq-uni321} and 
\eqref{eq-uni34}, we have 
\[
\|\eta(t) - A e^{- e_{\omega}t} 
\mathcal{Y}_{\omega, -}\|_{H^{1}} 
\leq C( e^{- \gamma_{2}t} 
+ e^{- \frac{\gamma_{1} 
+ \gamma_{2}}{2} t})
\leq C e^{- \frac{\gamma_{1} 
+ \gamma_{2}}{2} t}. 
\] 
Putting $\eta_0(t) 
:= \eta(t) - A e^{- e_{\omega}t} 
\mathcal{Y}_{\omega, -}$, 
we find that $\eta_0$ 
satisfies \eqref{eq-uni2} 
with $\gamma_{1}$ replaced by 
$\frac{\gamma_{1} + \gamma_{2}}{2} (> \gamma_{1})$. 
By iterating the above 
argument, 
we obtain 
\[
\|\eta_0(t)
\|_{H^{1}} \leq C_{1} 
e^{- \gamma_{2}^{-}}. 
\] 
From Lemma \ref{uni-lem-4}, 
we have 
$\|\langle \nabla \rangle 
\eta_0
\|_{St(t, \infty)} \leq C_{2} 
e^{- \gamma_{2}^{-} t}$. 
This completes the proof. 
\end{proof} 
\subsection{Proof of 
Proposition \ref{prop-unique}}
In this subsection, we give the 
proof of Proposition \ref{prop-unique}. 
First, we recall the 
following nonlinear estimates, which 
are obtained in Duyckaerts and Merle~\cite[Lemma 5.6]{DM2009} 
(see also Ardila and Murphy~\cite[Lemma 6.2]{AM2022}): 
\begin{lemma}[Nonlinear estimates]
\label{lem-nonest1}
\begin{equation*}
\begin{split}
\|N_{\omega, 1}(f) - N_{\omega, 1}(g)
\|_{L_{x}^{\frac{6}{5}}} 
\leq C\|f - g\|_{L_{x}^{\frac{18}{5}}}
& 
(\|Q_{\omega}\|_{L_{x}^{\frac{18}{5}}}
\|f\|_{L_{x}^{\frac{18}{5}}} + 
\|Q_{\omega}\|_{L_{x}^{\frac{18}{5}}}
\|g\|_{L_{x}^{\frac{18}{5}}} 
\\ & 
+ \|f\|_{L_{x}^{\frac{18}{5}}}^{2} + 
\|g\|_{L_{x}^{\frac{18}{5}}}^{2}), 
\end{split}
\end{equation*}
\begin{equation*}
\|N_{\omega, 1}(f) - N_{\omega, 1}(g)
\|_{L_{x}^{\frac{30}{23}}} 
\leq C\|f - g\|_{L_{x}^{\frac{30}{11}}}
(\|Q_{\omega}\|_{L_{x}^{5}}
\|f\|_{L_{x}^{5}} + 
\|Q_{\omega}\|_{L_{x}^{5}}
\|g\|_{L_{x}^{5}} 
+ \|f\|_{L_{x}^{5}}^{2} + 
\|g\|_{L_{x}^{5}}^{2}), 
\end{equation*}
\begin{equation*} 
\|N_{\omega, 2}(f) - N_{\omega, 2}(g)
\|_{L_{x}^{\frac{6}{5}}} 
\leq C\|f - g\|_{L_{x}^{6}}
(
\|Q_{\omega}\|_{L_{x}^{6}}^{3}
\|f\|_{L_{x}^{6}} + 
\|Q_{\omega}\|_{L_{x}^{6}}^{3}
\|g\|_{L_{x}^{6}} 
+ \|f\|_{L_{x}^{6}}^{4} + 
\|g\|_{L_{x}^{6}}^{4}), 
\end{equation*}
\begin{equation*} 
\|N_{\omega, 2}(f) - N_{\omega, 2}(g)
\|_{L_{x}^{\frac{30}{23}}} 
\leq C\|f - g\|_{L_{x}^{\frac{30}{11}}}
(
\|Q_{\omega}\|_{L_{x}^{10}}^{3}
\|f\|_{L_{x}^{10}} + 
\|Q_{\omega}\|_{L_{x}^{10}}^{3}
\|g\|_{L_{x}^{10}} 
+ \|f\|_{L_{x}^{10}}^{4} + 
\|g\|_{L_{x}^{10}}^{4}). 
\end{equation*}
Let $I$ be a finite time interval and 
$f, g$ be functions in $W(I)$ such that 
$\langle \nabla \rangle f$ 
and $\langle \nabla \rangle g$ 
are in $V(I)$. 
Then, we have 
\begin{equation}\label{non-est-81}
\begin{split}
\|\langle \nabla 
\rangle N_{\omega, 1}(f) - \langle \nabla 
\rangle N_{\omega, 1}(g)
\|_{N(I)} 
& \leq C \|\langle \nabla 
\rangle (f - g)\|_{V_{3}(I)} 
\cdot \biggl[|I|^{\alpha_{1}}
(\|\langle \nabla 
\rangle f\|_{V_{3}(I)} \\
& \qquad 
+ \|\langle \nabla 
\rangle g\|_{V_{3}(I)}) + 
\|\langle \nabla 
\rangle f\|_{V_{3}(I)}^{2} 
+ \|\langle \nabla 
\rangle g\|_{V_{3}(I)}^{2}\biggl],  
\end{split}
\end{equation}
\begin{equation}\label{non-est-82}
\begin{split}
\|\langle \nabla 
\rangle N_{\omega, 2}(f) - 
\langle \nabla \rangle N_{\omega, 2}(g)
\|_{N(I)} 
& \leq C \|\langle \nabla 
\rangle (f - g)\|_{V_{3}(I)} 
\cdot \bigg[|I|^{\alpha_{2}}
(\|\langle \nabla 
\rangle f\|_{V(I)} \\
& \qquad 
+ \|\langle \nabla 
\rangle g\|_{V(I)}) + 
\|\langle \nabla 
\rangle f\|_{V(I)}^{4} 
+ \|\langle \nabla 
\rangle g\|_{V(I)}^{4}\biggl],  
\end{split}
\end{equation}
where $\alpha_{1}$ and $\alpha_{2}$ are some positive constants. 
\end{lemma}
We also recall the following estimate of 
$\theta$.
\begin{lemma}\label{lem-estR2}
There exists $C>0$ 
such that 
for any $t\in [0, T_{X}]$, one has  
\begin{equation} \label{eq-ejec22}
\biggl|\frac{d \theta}{d t} 
- \omega \biggl| 
\leq C \|\eta(t)\|_{H^{1}}^{2}, 
\end{equation}
\end{lemma}
See \cite[(5.28)]{AIKN2021} for 
the proof of Lemma \ref{lem-estR2}. 
We are now in a position to prove 
Proposition \ref{prop-unique}. 
\begin{proof}[Proof of 
Proposition \ref{prop-unique}]
We divide the proof into 4 steps. 

\textbf{(Step 1). 
}
Let $\psi = e^{i \theta(t)} 
(Q_{\omega}(x) + \eta(t, x))$. 
It follows from 
\eqref{eq-energy7},  
\eqref{eq-mdist1} and the 
assumption \eqref{eq-uni1} that  
$\|\eta(t)\|_{H^{1}} \leq 
C e^{- c t}$ 
for $t \geq t_{0}$. 
We see from 
\eqref{eq-ejec22} that 
	\[
	\biggl|(\theta(t) - \omega t) 
	- (\theta(s) - \omega s)\biggl| 
	\leq 
	\int_{s}^{t}\biggl| 
	\frac{d \theta}{d\tau} 
	- \omega\biggl| d\tau 
	\leq \int_{s}^{t}
	C \|\eta(\tau)\|_{H^{1}}^{2} 
	d \tau 
	\leq C e^{- 2cs}. 
	\]
Thus, there exists 
$\theta_0 \in \R$
such that $\lim_{t \to \infty} 
(\theta(t) - \omega t) = \theta_0$. 
From the above inequality, we obtain 
  \begin{equation}\label{eq-uni59}
  |\theta(t) - \omega t - \theta_{0}| 
  \leq C e^{- 2 c t}
  \end{equation}
Let $\psi(t, x) = e^{i (\omega t 
+ \theta_0)} 
(Q_{\omega}(x) + 
\eta_{0}(t, x))$.
We claim that 
$\eta_0$ satisfies 
\begin{equation}
\label{eq-uni58}
\|\eta_0(t)\|_{H^{1}} \leq 
C e^{- ct} 
\qquad (t \geq t_{0}) 
\end{equation}
for some $c, C > 0$. 
Observe that 
	\[
	\eta_0(t, x) 
= e^{- i(\omega t + \theta_0)}\psi(t, x) 
- Q_{\omega}(x)
= e^{i (\theta(t) - \omega t - 
\theta_0)} 
(Q_{\omega}(x) + \eta(t, x)) - 
Q_{\omega}(x).
	\] 
This yields with \eqref{eq-uni59} and 
$\|\eta(t)\|_{H^{1}} \leq 
C e^{- c t}$ for $t \geq t_{0}$ that 
	\[
	\begin{split}
	\|\eta_0(t)\|_{H^{1}} 
	& = 
	\left\| 
	e^{i (\theta(t) - \omega t - 
\theta_0)} 
(Q_{\omega} + \eta(t)) - 
Q_{\omega}
	\right\|_{H^{1}} \\
	& 
	\leq \|\eta(t)\|_{H^{1}} 
	+ C \biggl|\theta(t) - \omega t - 
\theta_0 \biggl| 
\|Q_{\omega}\|_{H^{1}} \\
& \leq C e^{-ct} + C e^{-2ct} 
\leq C e^{-ct}. 
\end{split}
	\]
Thus, \eqref{eq-uni58} holds. 

\textbf{(Step 2). 
}
We claim that 
$\eta_0$ satisfies 
\begin{equation}
\label{eq-uni38}
\|\eta_0(t)\|_{H^{1}} \leq 
C e^{- e_{\omega}^{-} t} 
\qquad (t \geq t_{0}). 
\end{equation}
By \eqref{Stri-est-3} and 
Lemmas \ref{uni-lem-4}, \ref{lem-nonest1} 
with $f = \eta_0$ and $g = 0$, 
we obtain 
\[
\|N_{\omega}(\eta_0(t))
\|_{L_{x}^{\frac{6}{5}}} 
+ \|\langle \nabla \rangle 
N_{\omega}(\eta_0)\|_{N(t, \infty)} 
\leq Ce^{-2 c_{0}t}.   
\]
Then, applying Proposition 
\ref{pro-uni1} with $\gamma_{1} 
 = c_{0}$ and $\gamma_{2} 
 = 2 c_{0}$,  
$g = N_{\omega}(\eta_{0}(t))$, 
we have 
\[
\|\eta_0(t)\|_{H^{1}} 
\leq C (e^{-e_{\omega} t} 
+ e^{- \frac{3}{2} c_{0} t} 
). 
\] 
If $e_{\omega} \leq 
\frac{3}{2} c_{0}$, 
\eqref{eq-uni38} 
holds. 
If not, we get 
$\|\eta_0(t)\|_{H^{1}} \leq 
C e^{- \frac{3}{2} c_{0} t}$. 
Then, we can verify that 
\eqref{eq-uni38} holds by 
iterating the above argument.

\textbf{(Step 3). 
}
Applying Proposition 
\ref{pro-uni1} \textrm{(ii)} again 
with $\gamma_{1} 
= e_{\omega}^{-}, 
\gamma_{2} 
= 2e_{\omega}^{-} 
(> e_{\omega})$ and  
$g = N_{\omega}
(\eta_0(t))$, 
we find that there 
exists $A_0 \in \R$ such that 
\begin{equation}\label{eq-uni37}
\|\eta_0(t) - A_0 
e^{- e_{\omega} t} 
\mathcal{Y}_{\omega, -}\|
_{H^{1}} + 
\|\langle \nabla \rangle 
(\eta_0(t) - A_{0} 
e^{- e_{\omega} t} 
\mathcal{Y}_{\omega, -})\|
_{St(t, \infty)} 
\leq C e^{- 2 e_{\omega}^{-}t}. 
\end{equation}
Let $U_{\omega}^{A}$ be the solution 
constructed in Proposition 
\ref{prop-spec} for each 
$A \in \R$. 
We write $U_{\omega}^{A} 
= e^{i \omega t} (Q_{\omega} 
+ \eta^{A}(t))$. 
We claim that 
for any $\gamma > 0$, 
\begin{equation}
\label{eq-uni39}
\|\eta_0(t) 
- \eta^{A_{0}}(t)\|_{H^{1}} 
+ \|\langle \nabla \rangle 
(\eta_0(t) - \eta^{A_{0}}(t))
\|_{St(t, \infty)} \leq 
C e^{- \gamma t} 
\qquad (t \geq t_{0})
\end{equation}
for some $C > 0$. 
Observe from Proposition \ref{prop-spec} 
that  
$\eta^{A_0}$ satisfies
~\footnote{Note that 
$\mathcal{V}_{1, \omega}^{A_{0}} = 
e^{- e_{\omega} t} \mathcal{Z}_{1, 
\omega}^{A_{0}} 
= A_{0} e^{- e_{\omega} t} 
\mathcal{Y}_{\omega, -}$}
\begin{equation} \label{eq-uni56}
\|\eta^{A_0} - A_0 
e^{- e_{\omega}t} 
\mathcal{Y}_{\omega, -}\|
_{H^{1}} 
+ \|\langle \nabla \rangle 
\left(\eta^{A_0} - 
A_0 e^{- e_{\omega}t} 
\mathcal{Y}_{\omega, -} \right)\|
_{St(t, \infty)} \leq 
C e^{-\frac{3}{2} e_{\omega} t} 
\qquad (t \geq t_{0}). 
\end{equation}
It follows from 
\eqref{eq-uni37} that 
\eqref{eq-uni39} holds 
with $\gamma 
= \frac{3}{2} e_{\omega}$. 
We can easily verify that 
$\eta_0 
- \eta^{A_0}$ satisfies 
\[
\p_{t} (\eta_0 
- \eta^{A_0}) 
+ i \mathcal{L}_{\omega} 
(\eta_0 
- \eta^{A_0}) = 
N_{\omega}(\eta_0) - 
N_{\omega}(\eta^{A_0}). 
\]
Then, by 
Lemmas \ref{uni-lem-4} and \ref{lem-nonest1} with 
$f = \eta_0(t)$ 
and $g = \eta^{A_0}(t)$, 
we have 
\begin{equation*}
\begin{split}
\|N_{\omega}(\eta_0(t)) - 
N_{\omega}(\eta^{A_0}(t))
\|_{L_{x}^{\frac{6}{5}}}
+ \|\langle \nabla \rangle N_{\omega}(\eta_0) 
- \langle \nabla \rangle N_{\omega}(\eta^{A_0})
\|_{N(t, \infty)} 
\leq C 
e^{- \frac{5}{2} e_{\omega} t} 
\qquad (t \geq t_{0}).  
\end{split}
\end{equation*}
Then, by Proposition 
\ref{pro-uni1} \textrm{(ii)} 
with $\gamma_{2} 
= 2 e_{\omega }$, 
there exists $A_{1} \in \R$ such that 
$\eta_0 - \eta^{A_0} = A_{1} 
e^{- e_{\omega} t} \mathcal{Y}_{\omega, -} 
+ w_{1}(t)$ 
with  
\begin{equation} \label{eq-uni57}
\|w_{1}\|_{H^{1}} 
+ \|\langle \nabla \rangle 
w_{1} \|_{St(t, \infty)} \leq 
C e^{- \frac{7}{4} 
e_{\omega}t} 
\qquad (t \geq t_{0}). 
\end{equation}
It follows from \eqref{eq-uni39} with $\gamma = \frac{3}{2} e_{\omega}$ 
and \eqref{eq-uni57} that 
  \[
  |A_{1}|e^{- e_{\omega} t}\| \mathcal{Y}_{\omega, -}\|_{H^{1}} 
  \leq \|\eta_0 - \eta^{A_0}\|_{H^{1}}
  + \|w_{1}(t)\|_{H^{1}} \leq C e^{- \frac{3}{2} 
e_{\omega}t}, 
  \]
which implies $A_{1} = 0$. 
Thus, we see from 
\eqref{eq-uni57} that 
\eqref{eq-uni39} holds 
for $\gamma = \frac{7}{4} e_{\omega}$.
Iterating this argument, we see that 
\eqref{eq-uni39} holds. 

\textbf{(Step 4). 
} 
We derive a conclusion. 
Using \eqref{eq-uni39} with 
$\gamma = (k_{0} + 1) e_{\omega}$ 
and \eqref{const-eq1} with $k = k_{0}$, 
we see that 
	\[
	\begin{split}
	& \quad 
	\|\langle \nabla \rangle
(e^{- i \theta_{0}} \psi - e^{i \omega t} 
(Q_{\omega} 
+ \mathcal{V}_{k_{0}, \omega}^{A_{0}}))
\|_{St(t, \infty)} \\
& \leq 
\|\langle \nabla \rangle
(e^{- i \theta_{0}} \psi - U_{\omega}^{A_{0}})
\|_{St(t, \infty)}
+ 
\|\langle \nabla \rangle
(U_{\omega}^{A_{0}} - e^{i \omega t} 
(Q_{\omega} 
+ \mathcal{V}_{k_0, \omega}^{A_0}))
\|_{St(t, \infty)} \\
& \leq 
\|\langle \nabla \rangle
(\eta_0 - \eta^{A_{0}})
\|_{St(t, \infty)} 
+ e^{- (k_{0} + \frac{1}{2}) e_{\omega} 
t} \\
& \leq C e^{- (k_{0}+1) e_{\omega}t} 
+  e^{- (k_{0} + \frac{1}{2}) e_{\omega} 
t} \leq 
 e^{- (k_{0} + \frac{1}{2}) 
 e_{\omega} t}. 
	\end{split}
	\]
Thus, from the uniqueness of the 
solution satisfying \eqref{const-eq1} 
(see Proposition \ref{prop-spec}), 
we find that $\psi = 
U^{A_0}_{\omega}$. 
\end{proof}

\subsection{Proof of Theorem 
\ref{main-class}} 
We are now in a position to prove 
Theorem \ref{main-class}
\begin{proof}[Proof of 
Theorem \ref{main-class}] 
We divide the proof 
into two steps. 

\textbf{(Step 1). 
}
It follows from \eqref{const-eq1-1} 
and \eqref{def-spe} that 
	\[
	\begin{split}
	Q_{\omega}^{+}(t) 
	& = e^{- i \omega t_0} 
	U^{+1} (t + t_0, x) \\
	& = e^{i \omega t} Q_{\omega} 
	+ e^{- e_{\omega} t_0} 
	e^{(i \omega - e_{\omega})t} 
	\mathcal{Y}_{\omega, -} 
	+ O(e^{- \frac{3}{2} 
	e_{\omega} t}) 
	\qquad \mbox{in 
	$H^{1}(\R^{3})$}.
	\end{split}
	\]
Fix $A >0$. 
Let $t_1 = - t_0 - 
\frac{1}{e_{\omega}} \log A$. 
This yields that 
	\[
	Q_{\omega}^{+}(t + t_1) 
	= e^{i \omega t_{1}} 
	e^{i \omega t} Q_{\omega} 
	+ e^{- e_{\omega} t_0} 
	e^{(i \omega - e_{\omega})t} 
	e^{(i \omega - e_{\omega}) t_1}
	\mathcal{Y}_{\omega, -} 
	+ O(e^{- \frac{3}{2} 
	e_{\omega} t}) 
	\qquad \mbox{in 
	$H^{1}(\R^{3})$}.
	\]
From \eqref{const-eq1-1} and 
$e^{- e_{\omega} t_{1}} = A e^{e_{\omega} t_{0}}$, 
we obtain 
	\begin{equation} 
	\label{proof-meq1}
	\begin{split}
	e^{- i \omega t_1} 
	Q_{\omega}^{+}(t + t_1) 
	& = e^{i \omega t} Q_{\omega} 
	+ A e^{(i \omega - e_{\omega}) 
	t} \mathcal{Y}_{\omega, -} 
	+ O(e^{- \frac{3}{2} 
	e_{\omega} t}) \\
    & = e^{i \omega t} 
    (Q_{\omega} + A e^{- e_{\omega}t} 
    \mathcal{Y}_{\omega, -}) 
    + O(e^{- \frac{3}{2} e_{\omega} t})\\
	& = U^{A}_{\omega}
	+ O(e^{- \frac{3}{2} 
	e_{\omega} t}) 
	\qquad \mbox{in 
	$H^{1}(\R^{3})$}.
	\end{split}
	\end{equation}
From this and \eqref{const-eq1-1}, we see that 
there exists $C_{1} > 0$ such that 
	\[
	\|e^{- i \omega t_1} 
	Q_{\omega}^{+}(t + t_1) 
	- e^{i \omega t} Q_{\omega}
	\|_{H^{1}} \leq 
	C_{1} 
	e^{- e_{\omega} t} 
	\qquad 
	\mbox{for $t > 0$}. 
	\]
This together with Proposition 
\ref{prop-unique} yields that 
there exists $\widetilde{A} \in \R$ and 
$\widetilde{\theta}_{0} \in \R$
such that $e^{- i \omega t_1} 
	Q_{\omega}^{+}(t + t_1) 
	= e^{i \widetilde{\theta}_{0}}
    U^{\widetilde{A}}_{\omega}$. 
By \eqref{proof-meq1}, 
we have $\widetilde{A} = A$ and 
$\widetilde{\theta}_{0} = 0$, 
which yields that 
	\begin{equation}
	\label{proof-meq2}
	U^{A}_{\omega} = e^{- i \omega t_1} 
	Q_{\omega}^{+}(t + t_1).
	\end{equation} 	
	
\textbf{(Step 2). }

Let $\psi$ be a solution to 
\eqref{nls} with $\psi|_{t = 0} 
= \psi_{0} \in \mathcal{BA}
_{\omega}$. 	
If $\mathcal{K}(\psi) = 0$, 
then $\psi$ is the ground state 
of $m_{\omega}$. 
From the uniqueness 
of the ground state (see 
Proposition \ref{thm-lu} 
\textrm{(i)}), we see that 
$\psi(t, x) = 
e^{i \theta + i \omega t} Q_{\omega} 
(x)$ for some $\theta \in \R$. 

Assume that $\mathcal{K} 
(\psi) > 0$ and $\psi$ does not 
scatter for positive time.  
By Proposition \ref{prop-conv-p}, 
there exist constants $C, c >0$ such that 
	\[
	\text{dist}_{H^{1}} 
(\psi(t), \mathcal{O}(Q_{\omega})) 
\leq Ce^{- ct} 
	\qquad \mbox{for $t > 0$}. 
	\] 
Hence, $\psi(t)$ 
satisfies the assumption of 
Proposition \ref{prop-unique} and 
$\mathcal{K}(\psi) > 0$, 
which shows that $\psi= 
U^{A}_{\omega}$ for some $A > 0$. 
Thus, we see from \eqref{proof-meq2} 
that 
\textrm{(iii)} holds. 

Combining 
Propositions \ref{prop-conv-n} and 
\ref{prop-unique}, 
we can prove \textrm{(i)} 
by a similar argument of \textrm{(ii)}. 
 
\end{proof}
\appendix
\section{Proof of Proposition \ref{conv-gs}}\label{conv}
This appendix is devoted to the  
proof of Proposition \ref{conv-gs}. 
For each $\lambda > 0$ and $u \in H^{1}(\R^{3})$, 
we define 
  \[
  T_{\lambda}u(\cdot) 
  := \lambda^{\frac{3}{2}}u(\lambda \cdot). 
  \]
 We can easily find that for any 
 $u \in H^{1}(\R^{3}) \setminus \{0\}$, 
 there exists $\lambda(u) > 0$ such that 
   \begin{equation} \label{eq-A3}
   \mathcal{K}(T_{\lambda} u) 
   \begin{cases}
   > 0 & \quad \mbox{for $0 < \lambda 
   < \lambda(u)$}, \\
   = 0 & \quad \mbox{for $\lambda 
   = \lambda(u)$}, \\
   < 0 & \quad \mbox{for $\lambda 
   > \lambda(u)$}.  \\
   \end{cases}
   \end{equation}
By a standard argument, 
one has 
\begin{equation}\label{vari-v2}
\begin{split}
m_{\omega}
=
\inf\big\{ \mathcal{J}_{\omega}(u) \colon 
u \in H^{1}(\R^{3})\setminus \{0\}, \ \mathcal{K}(u)\le 0
\big\},
\end{split}
\end{equation}
where 
\begin{align}
\label{funct-J}
\mathcal{J}_{\omega}(u)
&:=\mathcal{S}_{\omega}(u)
-\frac{1}{2}\mathcal{K}(u) 
= \frac{\omega}{2}\|u\|_{L^{2}}^{2} 
+ \frac{1}{8}\|u\|_{L^{4}}^{4}
+ \frac{1}{3}\|u\|_{L^{6}}^{6}.
\end{align}
Since the ground state 
$Q_{\omega} \in H^{1}(\R^{3}) \setminus \{0\}$ exists 
for $\omega \in (0, \omega_{c})$ (see Theorem 
\ref{thm-0}),  
we see that $m_{\omega} = J_{\omega} (Q_{\omega}) > 0$.  In addition, we have 
  \begin{equation} \label{eq-A2}
  m_{\omega} < 
\frac{\sigma^{\frac{3}{2}}}{3} 
\qquad \mbox{for $\omega 
\in (0, \omega_{c})$}, 
  \end{equation}
where 
  \begin{equation}\label{eq-A1}
  \sigma := \inf\left\{ 
  \|\nabla u\|_{L^{2}}^{2} \colon 
  u \in \dot{H}^{1}(\R^{3}) \; 
  \mbox{with} \; 
  \|u\|_{L^{6}} = 1
  \right\}. 
  \end{equation}
See \cite[Theorem 1.4]{AIKN}. 
To prove Proposition \ref{conv-gs}, 
we need the following lemmas:   
\begin{lemma}[Brezis and Lieb \cite{Brezis-Lieb}]\label{thm-bl}
Let $\{u_{n}\}$ be a bounded sequence in $H^{1}(\R^{3})$ such that 
\begin{equation*}
\lim_{n\to \infty}u_{n}(x)=u_{\infty}(x) 
\qquad 
\mbox{almost all $x\in \R^{3}$}
\end{equation*}
for some function $u_{\infty} \in H^{1}(\R^{3})$. Then, for any $2 \le r \le 6$, 
\begin{equation*}
\lim_{n\to \infty}
\int_{\R^{3}} 
\big| |u_{n}|^{r}-|u_{n}-u_{\infty}|^{r}-|u_{\infty}|^{r} \big| \,dx =0
\end{equation*}
and 
\begin{equation*}
\lim_{n\to \infty} 
\int_{\R^{3}} 
\big| |\nabla u_{n}|^{2}-|\nabla 
\{ u_{n}-u_{\infty}\}|^{2}-|\nabla 
u_{\infty}|^{2} \big| 
\,dx =0.
\end{equation*}
\end{lemma}

\begin{proof}[Proof of Proposition \ref{conv-gs}]
First, we obtain a 
$H^{1}$-boundedness of 
the sequence $\{u_{n}\}$. 
For sufficiently large 
$n \in \mathbb{N}$, we have 
\begin{equation*}
\begin{split}
2 m_{\omega} + 1
\geq 
|\mathcal{S}_{\omega}(u_{n})| 
+ \frac{1}{3} 
|\mathcal{K}(u_{n})| 
& 
\geq 
|\mathcal{S}_{\omega}(u_{n})
- \frac{1}{3} \mathcal{K}(u_{n})| 
\\
& = 
\frac{1}{6} 
\|\nabla u_{n}\|_{L^{2}}^{2} 
+ \frac{\omega}{2} 
\|u_{n}\|_{L^{2}}^{2}
+ \frac{1}{6} 
\|u_{n}\|_{L^{6}}^{6}. 
\end{split}
\end{equation*} 
Therefore, we see that the sequence 
$\{u_{n}\}$ is bounded in 
$H^{1}(\R^{3})$. 

Then, up to a subsequence, there 
exists $Q_{\infty} \in H^{1}(\R^{3})$ 
such that $\lim_{n \to \infty} 
u_{n} = Q_{\infty}$ weakly in 
$H^{1}(\R^{3})$. 
We shall show that $Q_{\infty} 
\neq 0$. 
Suppose to the contrary that 
$Q_{\infty} \equiv 0$. 
From the compactness of 
the embedding 
$H_{\text{rad}}^{1}(\R^{3}) 
\subset L^{4}(\R^{3})$, we have
\begin{equation}\label{eq:A.9}
0 = 
\lim_{n\to \infty}
\mathcal{K}(u_{n})
= \lim_{n\to \infty}\left\{ 
\left\| \nabla u_{n} 
\right\|_{L^{2}}^{2}
-\left\| u_{n}\right\|_{L^{6}}^{6}
\right\}.
\end{equation}
Suppose that $\lim_{n \to \infty}
\| \nabla u_n\|_{L^2} = 0$. 
Then, it follows from the Sobolev 
embedding that 
$\lim_{n \to \infty} 
\| u_n \|_{L^6} = 0$. 
This yields that 
\begin{equation*}
0 < \mathcal{J}_{\omega}(Q_{\omega}) = 
m_{\omega} = \lim_{n \to \infty} \mathcal{S}_{\omega}
(u_{n}) = 0, 
\end{equation*} 
which is absurd. 
Thus, by taking a subsequence, 
we may assume $\lim_{n\to\infty} \| \nabla u_n \|_{L^2} > 0$. 
Now, \eqref{eq:A.9} with 
the definition of $\sigma$ \eqref{eq-A1} gives us
\begin{equation*}
\lim_{n\to \infty}
\left\| \nabla u_{n}
\right\|_{L^{2}}^{2}
\ge 
\sigma 
\lim_{n\to \infty}\left\| u_{n} 
\right\|_{L^{6}}^{2}
\ge 
\sigma \lim_{n\to \infty}
\left\| \nabla u_{n}
\right\|_{L^{2}}^{\frac{2}{3}}.
\end{equation*}
From this together with 
$\lim_{n\to\infty} 
\| \nabla u_n \|_{L^2} >0$ 
and \eqref{eq:A.9}, we have 
\begin{equation*}
\sigma^{\frac{3}{2}}
\le \lim_{n\to \infty}\left\| \nabla u_{n} \right\|_{L^{2}}^2 =  
\lim_{n \to \infty} 
\left\|u_{n} \right\|_{L^{6}}^{6}.
\end{equation*}
Hence, we see that 
\[
\begin{split}
m_{\omega}
=
\lim_{n\to \infty}\mathcal{S}
_{\omega}(u_{n})
=
\lim_{n\to \infty}\left\{
\mathcal{S}_{\omega}(u_{n}) 
- \frac{1}{2} \mathcal{K}(u_{n})
\right\} 
&\ge 
\lim_{n\to \infty}
\left\{ 
\frac{1}{8}\left\|u_{n}\right\|
_{L^{4}}^{4}
+
\frac{1}{3}\left\| u_{n} \right\|_{L^{6}}^{6}
\right\}
\\[6pt]
&\ge \frac{1}{3}\lim_{n\to \infty}
\left\|u_{n} \right\|_{L^{6}}^{6}
\ge \frac{\sigma^{\frac{2}{3}}}{3}, 
\end{split}
\]
which contradicts \eqref{eq-A2}. 
Thus, $Q_{\omega} \not\equiv 0$. 

It follows from 
$\lim_{n \to \infty} 
\mathcal{S}_{\omega}(u_{n}) 
= m_{\omega}$ and 
$\lim_{n \to \infty}
\mathcal{K}(u_{n}) = 0$ that 
\begin{equation}\label{eq:A.7} 
\lim_{n \to \infty} 
\mathcal{J}_{\omega}(u_{n}) 
= \lim_{n \to \infty} 
\left\{ 
\mathcal{S}_{\omega}(u_{n})
- \frac{1}{2}\mathcal{K}(u_{n})
\right\} = m_{\omega}
\end{equation}
Using Lemma \ref{thm-bl}, we have
\begin{align}
\label{eq:A.11}
\mathcal{J}_{\omega}(u_{n})
-
\mathcal{J}_{\omega}
(u_{n} -Q_{\infty})
-
\mathcal{J}_{\omega}
(Q_{\infty})
&=o_{n}(1),
\\[6pt]
\label{eq:A.12}
\mathcal{K}(u_{n}) 
- \mathcal{K}(u_{n} 
- Q_{\infty})
-\mathcal{K}(Q_{\infty})
&=o_{n}(1).
\end{align} 
Furthermore, \eqref{eq:A.11} together with \eqref{eq:A.7} 
and the positivity of $\mathcal{J}_{\omega}$ 
implies that
$\mathcal{J}_{\omega}(Q_{\infty}) 
\le m_{\omega}$.
We claim that 
$\mathcal{K}(Q_{\infty}) \leq 0$. 
Suppose to the contrary that $\mathcal{K}
(Q_{\infty})>0$. 
Then, it follows from 
$\lim_{n \to \infty} \mathcal{K}
(u_{n}) = 0$
and \eqref{eq:A.12} that $
\mathcal{K}(u_{n} - Q_{\infty}) 
<0$ 
for sufficiently large $n$. 
Hence, from \eqref{eq-A3}, 
we can take $\lambda_{n}\in 
(0,1)$ such that 
$\mathcal{K}
(T_{\lambda_{n}}
(u_{n} - Q_{\infty})) = 0$. 
Furthermore, we see from $0<\lambda_{n}<1$, 
and the definition of $\mathcal{J}_{\omega}$ 
that 
\begin{equation*}
\begin{split}
m_{\omega}
\le 
\mathcal{J}_{\omega} 
( T_{\lambda_{n}} (u_{n}- 
Q_{\infty})) 
& 
= \frac{\omega}{2} 
\|u_{n}- 
Q_{\infty}\|_{L^{2}}^{2} 
+ \frac{\lambda_{n}^{3}}{8} 
\|u_{n}- 
Q_{\infty}\|_{L^{4}}^{4} 
+ \frac{\lambda_{n}^{6}}{3} 
\|u_{n}- 
Q_{\infty}\|_{L^{6}}^{6} \\
& < 
\mathcal{J}_{\omega} 
(u_{n} - Q_{\infty}). 
\end{split}
\end{equation*}
In addition, it follows from 
\eqref{eq:A.7}, \eqref{eq:A.11} 
and $Q_{\infty} \neq 0$ that 
\begin{equation*}
\begin{split}
m_{\omega} < \mathcal{J}_{\omega} 
(u_{n} - Q_{\infty})
= \mathcal{J}_{\omega} 
(u_{n})
- \mathcal{J}_{\omega}
(Q_{\infty})+o_{n}(1)
= m_{\omega} -
\mathcal{J}_{\omega}(Q_{\infty}) 
+o_{n}(1) < m_{\omega}
\end{split}
\end{equation*}
for sufficiently large $n \in \mathbb{N}$, 
which is a contradiction. 
Thus, $\mathcal{K} 
(Q_{\infty}) \leq 0$.

Since $Q_{\infty} \not\equiv 0$ 
and $\mathcal{K} 
(Q_{\infty}) \leq 0$, 
it follows from \eqref{vari-v2} 
that 
\begin{equation}\label{eq:A.15}
m_{\omega}
\le \mathcal{J}_{\omega}
(Q_{\infty}). 
\end{equation}
Moreover, it follows from 
the weak lower semicontinuity that 
\begin{equation}\label{eq:A.16}
\mathcal{J}_{\omega}
(Q_{\infty}) 
\le 
\liminf_{n\to \infty}
\mathcal{J}_{\omega}(u_{n})
\le m_{\omega}.
\end{equation} 
Combining (\ref{eq:A.15}) and 
(\ref{eq:A.16}), we obtain $
\mathcal{J}_{\omega}
(Q_{\infty})=m_{\omega}$. 
Thus, we have proved that 
$Q_{\infty}$ 
is a minimizer for $m_{\omega}$. 
Then, from the uniqueness of the ground state 
(see Proposition \ref{thm-lu} \textrm{(i)}), 
there exists $\theta \in \R$ such that 
$Q_{\infty} 
= e^{i \theta} Q_{\omega}$. 
Thus, we see that $\mathcal{S}_{\omega}(Q_{\infty}) 
= m_{\omega}$ and $\mathcal{K}(Q_{\infty}) = 0$. 
It follows from $\mathcal{J}_{\omega} (Q_{\infty}) = m_{\omega} = 
\lim_{n \to \infty} \mathcal{J}_{\omega} (u_{n})$ and 
\eqref{eq:A.11} that 
$\lim_{n \to \infty} \mathcal{J}_{\omega} (u_{n} - Q_{\infty}) = 0$. 
This together with the 
H\"{o}lder inequality 
yields that $\lim_{n \to \infty} \|u_{n} - Q_{\infty}\|_{L^{q}} 
= 0$ for $2 \leq q \leq 6$. 
Then, since $\lim_{n \to \infty} 
\mathcal{S}_{\omega}(u_{n}) = m_{\omega} 
= S_{\omega}(Q_{\infty})$, 
we have 
  \[
  \begin{split}
  \lim_{n \to \infty} 
  \left\{\frac{1}{2} \|\nabla u_{n}\|_{L^{2}}^{2} 
  + \frac{\omega}{2} \|u_{n}\|_{L^{2}}^{2} \right\}
 & =  \lim_{n \to \infty} \left\{
 \mathcal{S}_{\omega}(u_{n}) 
 - \frac{1}{4} \|u_{n}\|_{L^{4}}^{4} 
 - \frac{1}{6} \|u_{n}\|_{L^{6}}^{6}
 \right\} \\
 & = \mathcal{S}_{\omega}(Q_{\infty}) 
 - \frac{1}{4} \|Q_{\infty}\|_{L^{4}}^{4} 
 - \frac{1}{6} \|Q_{\infty}\|_{L^{6}}^{6} \\
 & = \frac{1}{2} \|\nabla Q_{\infty}\|_{L^{2}}^{2} 
  + \frac{\omega}{2} \|Q_{\infty}\|_{L^{2}}^{2}. 
 \end{split}
  \]
This together with the weak convergence of  
$u_{n}$ to $Q_{\infty}$ in $H^{1}(\R^{3})$ 
implies that 
  \[
  \lim_{n \to \infty} 
u_{n} = Q_{\infty} = e^{i \theta} Q_{\omega} 
\qquad \mbox{strongly in 
$H^{1}(\R^{3})$}.
  \]
Thus, we infer that $\lim_{n \to \infty} 
\text{dist}_{H^{1}} 
(u_{n}, \mathcal{O}(Q_{\omega})) 
= 0$.

\end{proof}


\begin{thank}
M. H. was supported by JSPS KEKENHI Grant Number 
JP22J00787. 
H.K. was supported by JSPS KAKENHI 
Grant Number JP20K03706.
M.W was supported by JSPS KAKENHI Grant Number 22J10027. 
\end{thank}

\bibliographystyle{plain}

\vspace{24pt}

\noindent
Masaru Hamano
\\
Faculty of Science and Engineering, 
\\
Waseda University
\\
3-4-1 Okubo, Shinjuku-ku, Tokyo 169-8555, JAPAN
\\
E-mail: m.hamano3@kurenai.waseda.jp

\vspace{0.5cm}

\noindent
Hiroaki Kikuchi
\\
Department of Mathematics
\\
Tsuda University
\\
2-1-1 Tsuda-machi, Kodaira-shi, Tokyo 187-8577, JAPAN
\\
E-mail: hiroaki@tsuda.ac.jp

\vspace{0.5cm}

\noindent
Minami Watanabe
\\
Graduate school of Mathematics
\\
Tsuda University
\\
2-1-1 Tsuda-machi, Kodaira-shi, Tokyo 187-8577, JAPAN
\\
E-mail: m18mwata@gm.tsuda.ac.jp

\end{document}